\documentclass[english]{elsarticle}
\usepackage[T1]{fontenc}
\usepackage[latin9]{inputenc}
\usepackage{color}
\usepackage{array}
\usepackage{textcomp}
\usepackage{amsmath}
\usepackage{amsthm}
\usepackage{amssymb}
\usepackage{stackrel}
\usepackage{graphicx}

\makeatletter

\providecommand{\tabularnewline}{\\}
\newcommand{\lyxdot}{.}

\numberwithin{equation}{section}
\numberwithin{figure}{section}
\theoremstyle{plain}
\newtheorem{thm}{\protect\theoremname}
  \theoremstyle{plain}
  \newtheorem{lem}[thm]{\protect\lemmaname}
  \theoremstyle{definition}
  \newtheorem{defn}[thm]{\protect\definitionname}
  \theoremstyle{plain}
  \newtheorem{prop}[thm]{\protect\propositionname}
  \theoremstyle{plain}
  \newtheorem{cor}[thm]{\protect\corollaryname}

\usepackage{epsfig}

\makeatother

\usepackage{babel}
  \providecommand{\corollaryname}{Corollary}
  \providecommand{\definitionname}{Definition}
  \providecommand{\lemmaname}{Lemma}
  \providecommand{\propositionname}{Proposition}
\providecommand{\theoremname}{Theorem}

\begin{document}

\begin{frontmatter}{}

\title{\textcolor{black}{Accounting for the Role of Long Walks on Networks
via a New Matrix Function}}

\author{\textcolor{black}{Ernesto Estrada and Grant Silver}}

\address{\textcolor{black}{Department of Mathematics \& Statistics, University
of Strathclyde, 26 Richmond Street, Glasgow G1 1XQ, UK}}
\begin{abstract}
\textcolor{black}{We introduce a new matrix function for studying
graphs and real-world networks based on a double-factorial penalization
of walks between nodes in a graph. This new matrix function is based
on the matrix error function. We find a very good approximation of
this function using a matrix hyperbolic tangent function. We derive
a communicability function, a subgraph centrality and a double-factorial
Estrada index based on this new matrix function. We obtain upper and
lower bounds for the double-factorial Estrada index of graphs, showing
that they are similar to those of the single-factorial Estrada index.
We then compare these indices with the single-factorial one for simple
graphs and real-world networks. We conclude that for networks containing
chordless cycles\textemdash holes\textemdash the two penalization
schemes produce significantly different results. In particular, we
study two series of real-world networks representing urban street
networks, and protein residue networks. We observe that the subgraph
centrality based on both indices produce significantly different ranking
of the nodes. The use of the double factorial penalization of walks
opens new possibilities for studying important structural properties
of real-world networks where long-walks play a fundamental role, such
as the cases of networks containing chordless cycles.}
\end{abstract}

\end{frontmatter}{}

\section{\textcolor{black}{Introduction}}

\textcolor{black}{The study of large graphs and networks has become
an important topic in applied mathematics, computer sciences and beyond
\cite{Newman review,EstradaBook}. The role played by such large graphs
in representing the structural skeleton of complex systems\textemdash ranging
from social to ecological and infrastructural ones\textemdash has
triggered the production of many indices that try to quantify the
different structural characteristics of these networks \cite{EstradaBook,LucianoReview}.
Among those mathematical approaches used nowadays for studying networks,
matrix functions \cite{Function of matrices} of adjacency matrices
of graphs have received an increasing visibility due to their involvement
in the so-called }\textit{\textcolor{black}{communicability functions}}\textcolor{black}{{}
\cite{Estrada Higham,Communicability,Protein folding,Subgraph Centrality,Estrada Hatano Benzi,Katz centrality,Estrada index,Estrada index 1,Estrada index 2,Natural connectivity 1,Natural connectivity 2,Benzi 1,Benzi 2,Benzi 3,Ejov_1,Ejov_2}.
These functions characterize how much information flows between two
different nodes of a graph by accounting for a weighted sum of all
the routes connecting them. Here, a route is synonymous with a walk
connecting two nodes, which is a sequence of (not necessarily distinct)
consecutive vertices and edges in the graph. Then, the communicability
function between the nodes $p$ and $q$ is defined by the $p,q$-entry
of the following function of the adjacency matrix (see \cite{Estrada Higham,Estrada Hatano Benzi}
and references therein)}

\textcolor{black}{
\begin{equation}
G=\sum_{k=0}^{\infty}c_{k}A^{k},\label{eq:general matrix function}
\end{equation}
where the coefficients $c_{k}$ are responsible of giving more weight
to shorter walks. The most popular of these communicability functions
is the one derived from the scaling of $c_{k}=\frac{1}{k!}$, which
gives rise to the exponential of the adjacency matrix (see further
for definitions). This function, and the graph-theoretic invariants
derived from it, have been widely applied in practical problems covering
a wide range of areas. Just to mention a few, the communicability
function is used for studying real-world brain networks and the effects
of diseases on the normal functioning of the human brain \cite{Brain networks,brain networks_2}.
On the other hand, the so-called subgraph centrality \cite{Subgraph Centrality}\textemdash a
sort of self-communicability of a node in a graph\textemdash has been
used to detect essential proteins in protein-protein interaction networks
\cite{PIN_1,PIN_2}. The network bipartivity\textemdash a measure
derived from the use of the self-communicability\textemdash has found
applications ranging from detection of cracks in granular material
\cite{Tordesillas}, to the stability of fullerenes \cite{fullerene},
and transportation efficiency of airline networks \cite{airlines}.}

\textcolor{black}{A typical question when studying the structural
indices derived from (\ref{eq:general matrix function}) when using
$c_{k}=\frac{1}{k!}$ is whether or not we are penalizing the longer
routes in the graph too heavily (see Preliminaries for formal definitions)
\cite{Zooming in and out}. To understand this problem let us consider
the communicability function between the nodes $p$ and $q$ in the
graph: }

\textcolor{black}{
\begin{equation}
G_{pq}=\sum_{k=0}^{\infty}c_{k}\left(A^{k}\right)_{pq},\label{eq:general centrality}
\end{equation}
where $\left(A^{k}\right)_{pq}$ gives the number of routes of length
$k$ between these two nodes. Then, when we use $c_{k}=\frac{1}{k!}$
a route of length 2 is penalized by 1/2 and a walk of length 3 is
penalized by 1/6. However, a walk of length 5 is already penalized
by $1/120\approx0.008$, which could be seen as a very heavy penalization
for a relatively short walk between these two nodes. This means that
the longer walks connecting two nodes make a little contribution to
the communicability function. If we consider the function accounting
for the self-returning walks starting (and ending) at a given node
$G_{pp}$, a heavy penalization of longer walks means that this index
is mainly dependent on the degree of the corresponding node, i.e.,
the number of edges incident to it. That is, }

\textcolor{black}{
\begin{equation}
G_{pp}=1+c_{2}k_{p}+\sum_{k=3}^{\infty}c_{k}\left(A^{k}\right)_{pp},
\end{equation}
where $k_{p}$ is the degree of the node $p$. Then, the main question
here is to study whether using coefficients $c_{k}$ that do not penalize
the longer walks as heavily will reveal some structural information
of networks which is important in practical applications of these
indices.}

\textcolor{black}{Here we consider the use of a double-factorial penalization
1/$k!!$ \cite{DF} of walks as a way to increase the contribution
of longer walks in communicability-based functions for graphs and
real-world networks. The goal of this paper is two-fold. First, we
want to investigate whether this new penalization of walks produces
structural indices that are significantly different from the ones
derived from the factorial penalization. The other goal is to investigate
whether the information contained in longer walks is of significant
relevance for describing the structure of graphs and real-world networks.
While in the first case we can obtain analytical results that account
for the similarities and differences among the two penalization schemes,
in the second case we need to use some kind of indirect inference.
That is, we aim to explore some practical applications of the indices
derived from these two schemes and show whether or not there are significant
advantages when using one or the other for solving such practical
problems. In the current work we have strong evidences that the contributions
of long walks in networks is very important for such graphs containing
chordless cycles\textemdash also known as holes. In particular we
have considered a centrality index based on single- as well as on
double-factorial penalization of the walks, and observed that there
are significant differences in the ranking of the nodes when the graphs
contain such kind of topological features. Chordless cycles are ubiquitous
in certain scenarios, such as urban street networks and protein residue
networks, which are both studied here. In addition, these chordless
cycles are undesired features in certain networks like sensor networks,
mobile phone networks and other communication systems, where they
represent zones of no coverage of the signals.}

\section{\textcolor{black}{Preliminaries}}

\textcolor{black}{We consider in this work simple, undirected and
connected graphs $G=(V,E)$ with $n$ nodes (vertices) and $m$ edges.
A }\textit{\textcolor{black}{walk}}\textcolor{black}{{} of length $k$
in $G$ is a set of nodes $i_{1},i_{2},\ldots,i_{k},i_{k+1}$ such
that for all $1\leq l\leq k$, $(i_{l},i_{l+1})\in E$. A }\textit{\textcolor{black}{closed
walk}}\textcolor{black}{{} is a walk for which $i_{1}=i_{k+1}$. Let
$A$ be the adjacency operator on $\ell_{2}(V)$, namely $(Af)(p)=\sum_{q:{\rm dist}(p,q)=1}f(q)$.
For simple finite graphs $A$ is the symmetric adjacency matrix of
the graph. In the particular case of an undirected network as the
ones studied here, the associated adjacency matrix is symmetric, and
thus its eigenvalues are real. We label the eigenvalues of $A$ in
non-increasing order: $\lambda_{1}\geq\lambda_{2}\geq\ldots\geq\lambda_{n}$.
Since $A$ is a real-valued, symmetric matrix, we can decompose it
as }

\textcolor{black}{
\begin{equation}
A=U\Lambda U^{T},\label{eq:spectral decomposition}
\end{equation}
where $\Lambda$ is a diagonal matrix containing the eigenvalues of
$A$ and $U$ is the matrix containing the orthonormalized eigenvectors
$\mathbf{\overrightarrow{\psi}}_{i}$ associated with $\lambda_{i}$
as its columns. The graphs considered here are connected, therefore
$A$ is irreducible and from the Perron-Frobenius theorem we can deduce
that $\lambda_{1}>\lambda_{2}$ and that the leading eigenvector $\mathbf{\overrightarrow{\psi}}_{1}$,
which will sometimes be referred to as the }\textit{\textcolor{black}{Perron
vector}}\textcolor{black}{, can be chosen such that its components
$\mathbf{\mathbf{\overrightarrow{\psi}}}_{1,u}$ are positive for
all $u\in V$.}

\textcolor{black}{The degree of a node is the number of edges incident
to that node. The graph density is defined as }

\textcolor{black}{
\begin{equation}
d=\dfrac{2m}{n\left(n-1\right)},
\end{equation}
where $m$ is the number of edges in the graph. }

\textcolor{black}{The so-called 'exponential' communicability function
\cite{Communicability,Estrada Higham,Estrada Hatano Benzi} is defined
for a pair of nodes $p$ and $q$ on $G$ as}

\textcolor{black}{
\begin{equation}
G_{pq}=\sum_{k=0}^{\infty}\frac{\left(A^{k}\right)_{pq}}{k!}=\left(\exp\left(A\right)\right)_{pq}=\sum_{j=1}^{n}e^{\lambda_{j}}\mathbf{\mathbf{\overrightarrow{\psi}}}_{j,p}\mathbf{\overrightarrow{\psi}}_{j,q}.
\end{equation}
}

\textcolor{black}{The $G_{pp}$ terms of the communicability function
characterize the degree of participation of a node in all subgraphs
of the network, giving more weight to the smaller ones. Thus, it is
known as the }\textit{\textcolor{black}{subgraph centrality}}\textcolor{black}{{}
of the corresponding node \cite{Subgraph Centrality}. The global
structural index defined by}

\textcolor{black}{
\begin{equation}
EE\left(G\right)=tr\left(\exp\left(A\right)\right)=\sum_{j=1}^{n}e^{\lambda_{j}},
\end{equation}
is known as the }\textit{\textcolor{black}{Estrada index}}\textcolor{black}{{}
of the graph. The indices have been generalized by the use of a parameter
$\beta$ in the matrix function following the work of \cite{Temperature}.}

\textcolor{black}{
\begin{equation}
G_{pq}\left(\beta\right)=\sum_{k=0}^{\infty}\frac{\left(\beta^{k}A^{k}\right)_{pq}}{k!}=\left(\exp\left(\beta A\right)\right)_{pq}.\label{eq:general}
\end{equation}
}

\section{\textcolor{black}{Double-Factorial Penalization of Network Walks}}

\textcolor{black}{Let us start this section by recalling what the
double-factorial is. Let $k$ be a positive integer, then the double-factorial
$k!!$ is defined by}

\textcolor{black}{
\begin{equation}
k!!=\left\{ \begin{array}{cc}
k(k-2)(k-4)...3.1 & k\text{ odd }\\
k(k-2)(k-4)...4.2 & k\text{ even}\\
1 & k=-1,0.
\end{array}\right.
\end{equation}
}

\textcolor{black}{As a variation of the factorial $k!$ the double-factorial
appears very suitable for use as the penalization factor of the number
of walks of length $k$ in the definition of communicability functions.
Other functions have been used in the past for changing the heavy
penalization imposed by the single factorial. For instance, the use
of $c_{k}=\alpha^{k}$ where $0<\alpha<\lambda_{1}^{-1}$, has been
used since the introduction of the Katz index in 1953 \cite{Katz centrality}.
It is well-known that in this case the Eq. (\ref{eq:general}) converges
to the resolvent of the adjacency matrix, i.e., $\left[\left(I-\alpha A\right)^{-1}\right]_{pq}$.
Another choice of the coefficient $c_{k}$ in the Eq. (\ref{eq:general centrality})
is to consider $1/\left(k-t\right)!$ for some $t>0$ \cite{Zooming in and out}.
In this case the function (\ref{eq:general centrality}) converges
to \cite{Zooming in and out}:}

\textcolor{black}{
\begin{equation}
\left[A^{t}\left(I+Ae^{A}-e^{A}\right)\right]_{pq}.
\end{equation}
}

\textcolor{black}{The main problem with the two functions previously
mentioned is that they are parametric. In the first case we should
select the parameter $\alpha$ that is more appropriate for each individual
problem. It should be noticed that for very big networks, where $\lambda_{1}\gg1,$
the range of this parameter is very narrow leaving very little choice
for its variation. In the second case we also need to select the parameter
$t$ for each particular problem. Then, our consideration here is
the selection of a penalization which is not as heavy as the single
factorial but that does not contain any parameter. To see the main
differences and similarities with the other penalization discussed
before let us consider the terms $A^{k}/k!!$, where every walk of
length $k$ is penalized by $1/k!!$ and let us compare it with the
penalization of $1/k!,$ $\alpha^{k}$ and $1/\left(k-t\right)!$.
To give a simple example we consider a graph having $n=10$ nodes
and $m=40$ edges and show in Figure \ref{Figure3.11} the values
of $c_{k}\cdot tr\left(A^{k}\right)$ for $1\leq k\leq300$. As can
be seen in Figure \ref{Figure3.11} there are significant difference
among the three kinds of penalization of walks. The single factorial,
$\left(k-10\right)!$ and $0.01^{k}$ all display very similar behaviour,
with very quick decay for relatively small values of $k$. The use
of $0.1^{k}$ shows a smoother decay with the increase of $k$ (notice
that the plot is semi-log scale). However, for values $1\leq k\leq250$
the double-factorial penalizes the walks less heavily than this modified
factorial. As can be seen from this Figure, the double-factorial does
not penalize the long walks as heavily as the other penalization coefficients,
which may retain some important structural information of graphs and
networks. Hereafter, we will concentrate our analysis and comparison
between the double and the single-factorial penalization of walks
in graphs/networks.}

\begin{figure}
\begin{centering}
\textcolor{black}{\includegraphics[width=0.75\textwidth]{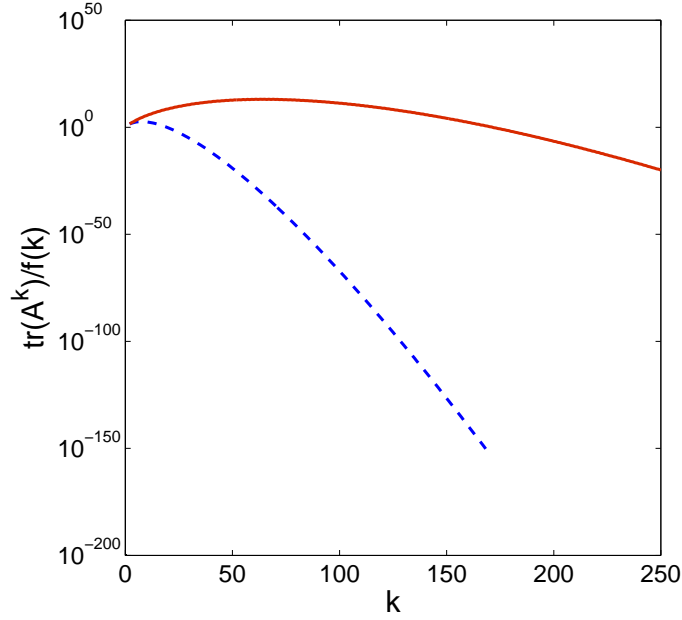}}
\par\end{centering}
\caption{Comparison of the decay of $c_{k}\cdot tr\left(A^{k}\right)$ for
different coefficients $c_{k}$ (see text for discussion). The $y$-axis
is expressed in logarithmic form.}

\label{Figure3.11}
\end{figure}

\textcolor{black}{As we are interested in defining matrix functions
that allow us to calculate several graph invariants, we start by proving
the following result.}
\begin{lem}
\textcolor{black}{Let $A$ be the adjacency matrix of a simple graph
$G=\left(V,E\right)$. Then}

\textcolor{black}{
\begin{equation}
\sum_{k=0}^{\infty}\frac{A^{k}}{k!!}=\frac{1}{2}\left[\sqrt{2\pi}\textrm{erf}\left(\frac{A}{\sqrt{2}}\right)+2I\right]\textrm{exp}\left(\frac{A^{2}}{2}\right),\label{eq:matrix function}
\end{equation}
where $I$ is the identity matrix and $\text{erf}(A)$ is the matrix
error function of $A$ \cite{erf}.}
\end{lem}
\begin{proof}
\textcolor{black}{Let us consider the spectral decomposition (\ref{eq:spectral decomposition}).
Then we can write}

\textcolor{black}{
\begin{eqnarray*}
\left(\sum_{k=0}^{\infty}\frac{A^{k}}{k!!}\right)_{pq} & = & \sum_{k=0}^{\infty}\sum_{j=1}^{n}\psi_{j,p}\psi_{j,q}\dfrac{\lambda_{j}^{k}}{k!!}\\
 & = & \sum_{j=1}^{n}\psi_{j,p}\psi_{j,q}\sum_{k=1}^{\infty}\dfrac{\lambda_{j}^{k}}{k!!}\\
 & = & \dfrac{1}{2}\sum_{j=1}^{n}\psi_{j,p}\psi_{j,q}\exp\left(\lambda_{j}^{2}/2\right)\left[\sqrt{2\pi}\textrm{erf}\left(\dfrac{\lambda_{j}}{\sqrt{2}}\right)+2\right],
\end{eqnarray*}
which can be written in the matrix form (\ref{eq:matrix function}),
proving the result.}
\end{proof}
\medskip{}

\textcolor{black}{From the computational point of view the main problem
for obtaining (\ref{eq:matrix function}) is provided by the calculation
of the matrix error function. In order to circumvent this difficulty
we make use here of the remarkable similarity between $\text{erf}(x)$
and $\tanh(x)$ (see Figure \ref{Figure3.2a}(a)). As can be seen
in Figure \ref{Figure3.2a}(a) there is a gap between the functions
in the interval $-2\leq x\leq2$. We can definitively improve the
similarity between the two function in the following way. The function
$\left[\text{erf}(x)-\tanh(kx)\right]$ is odd and so its integral
from $-\infty$ to $\infty$ is zero. Then, we will consider the integral}

\begin{equation}
\int_{0}^{\infty}\left[\text{erf}(x)-\tanh(kx)\right],\label{eq:integral diff}
\end{equation}
which we will make equal to zero as a way to minimize the difference
between the two functions. In other words, we will find the value
of $k$ for which (\ref{eq:integral diff}) is zero. Mathematically, 

\[
\lim_{a\rightarrow\infty}\int_{0}^{a}\left[\text{erf}(x)-\tanh(kx)\right]=0,
\]
which after integration becomes

\[
\lim_{a\rightarrow\infty}\left[a\text{erf}(a)-\frac{\text{ln(cosh}(ka))}{k}\right]=\frac{1}{\sqrt{\pi}}.
\]
Using the relation between the hyperbolic cosine and the exponential
we have

\[
\lim_{a\rightarrow\infty}\left[a\text{erf}(a)-\frac{\text{ln}(e^{ka}+e^{-ka})}{k}\right]+\frac{\text{ln}(2)}{k}=\frac{1}{\sqrt{\pi}}
\]

As $a$ grows to infinity, $e^{-ka}$ will vanish and $\text{erf}\left(a\right)=1$.
Then

\[
\lim_{a\rightarrow\infty}\left[a-\ln\left(e^{ka}\right)\right]+\frac{\text{ln}(2)}{k}=\frac{1}{\sqrt{\pi}},
\]
leading to

\[
\frac{\text{ln}(2)}{k}=\frac{1}{\sqrt{\pi}}.
\]
\textcolor{black}{Which gives us the result of $k=\sqrt{\pi}\text{ln}(2)$.
That is, $k=\sqrt{\pi}\ln(2)$ minimizes the gap between the two functions
as can be seen in Figure \ref{Figure3.2a}(b). Consequently, we define
the matrix function}

\textcolor{black}{
\begin{eqnarray}
D'\left(A\right) & = & \sum_{k=0}^{\infty}\frac{A^{k}}{k!!}=\frac{1}{2}\left[\sqrt{2\pi}\textrm{erf}\left(\frac{A}{\sqrt{2}}\right)+2I\right]\textrm{exp}\left(\frac{A^{2}}{2}\right),\\
 & \simeq & \frac{1}{2}\left[\sqrt{2\pi}\tanh\left(\frac{kA}{\sqrt{2}}\right)+2I\right]\textrm{exp}\left(\frac{A^{2}}{2}\right),
\end{eqnarray}
where}

\textcolor{black}{
\[
\tanh(kA)=\frac{e^{kA}-e^{-kA}}{e^{kA}+e^{-kA}}.
\]
}

\textcolor{black}{Hereafter, we define the function }

\textcolor{black}{
\begin{eqnarray}
D\left(A\right) & = & \frac{1}{2}\left[\sqrt{2\pi}\tanh\left(\frac{kA}{\sqrt{2}}\right)+2I\right]\textrm{exp}\left(\frac{A^{2}}{2}\right),
\end{eqnarray}
where we use $\tanh\left(kA\right)$ instead of $\text{erf}(A)$.
It represents }an approximate double-factorial or quasi-double-factorial
\textcolor{black}{function, but for the sake of simplicity we will
simply refer to it generalically as the double-factorial approach.
We then define the following indices that will be studied in this
work.}

\begin{figure}
\begin{centering}
\textcolor{black}{\includegraphics[width=0.5\textwidth]{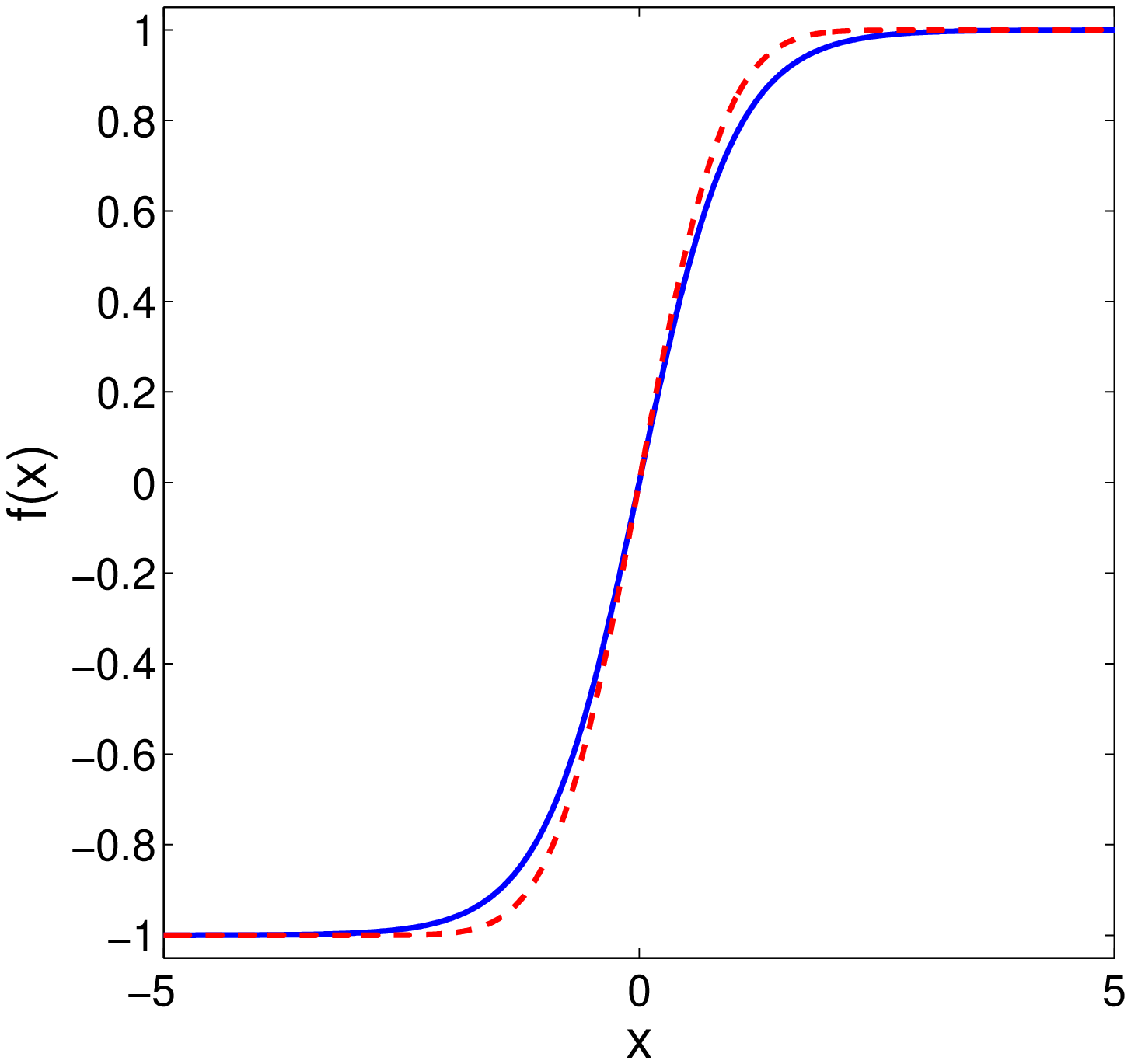}\includegraphics[width=0.5\textwidth]{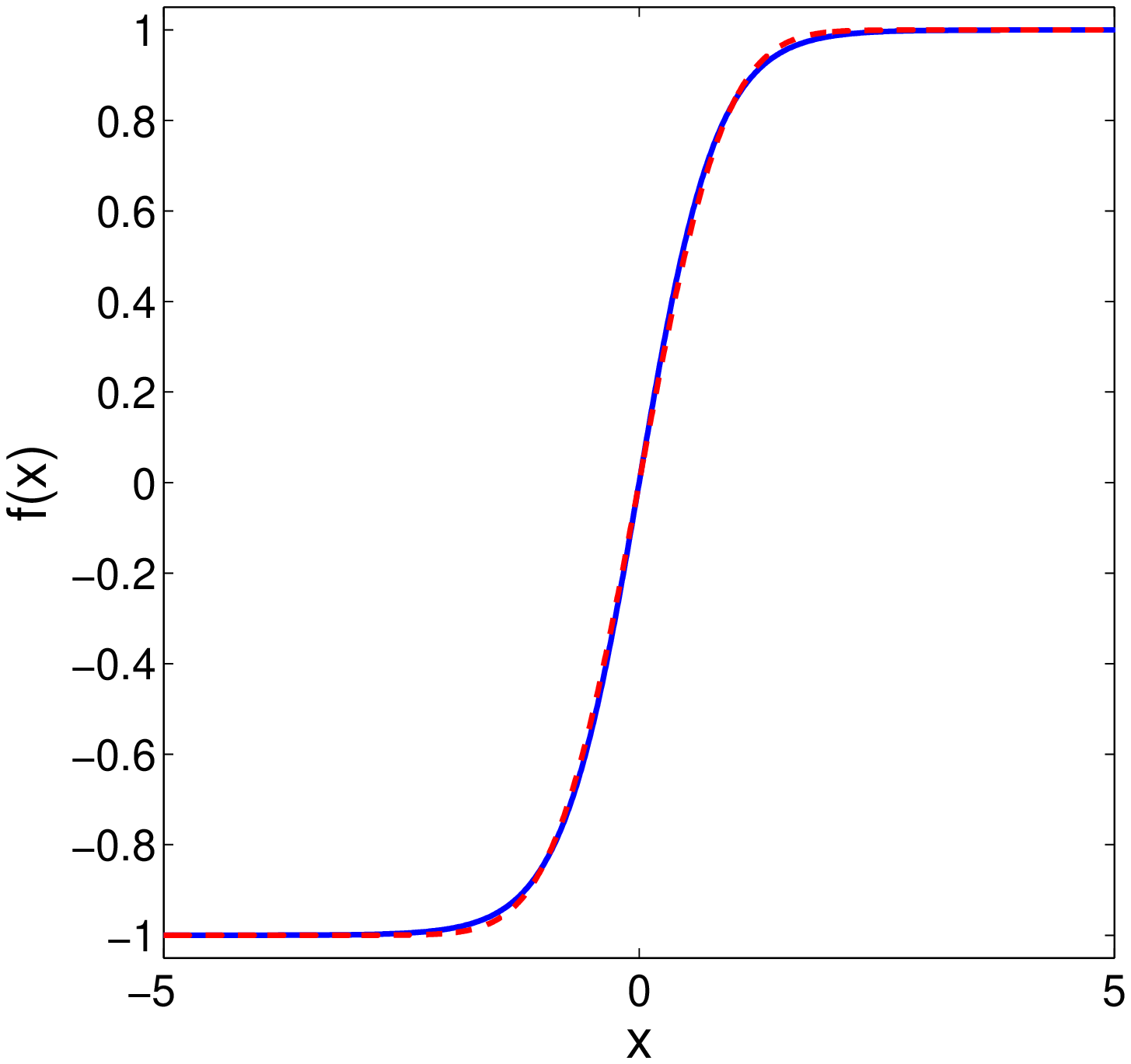}}
\par\end{centering}
\caption{\textcolor{black}{(a) Illustration of the similarities between $\text{erf}(x)$
(solid blue line) and $\tanh\left(x\right)$ (broken red line). (b)
Similar comparison between $\text{erf}(x)$ (solid blue line) and
$\tanh\left(kx\right)$ (broken red line) for $k=\sqrt{\pi}\log(2)$. }}

\label{Figure3.2a}
\end{figure}

\begin{defn}
\textcolor{black}{Let $p$ and $q$ be any two nodes of the graph
$G.$ The double-factorial communicability between these two nodes
is defined by $\Gamma_{pq}=\left(D\left(A\right)\right)_{pq}.$ Similarly,
the term $\Gamma_{pp}=\left(D\left(A\right)\right)_{pp}$ will be
called the double-factorial subgraph centrality of the node $p$ and
$\Gamma\left(G\right)=tr\left(D\left(A\right)\right),$ the double-factorial
Estrada index of $G.$ }

\textcolor{black}{The generalization of the new matrix function and
the indices derived from it lead naturally to considering the following
parameter $\beta\in\mathbb{R}$. That is, in general we can consider }
\end{defn}
\textcolor{black}{
\begin{eqnarray}
D\left(A,\beta\right) & = & \frac{1}{2}\left[\sqrt{2\pi}\tanh\left(\frac{k\beta A}{\sqrt{2}}\right)+2I\right]\textrm{exp}\left(\frac{\beta^{2}A^{2}}{2}\right),
\end{eqnarray}
and the corresponding indices $\Gamma_{pq}\left(\beta\right)=\left(D\left(A,\beta\right)\right)_{pq}$,
$\Gamma_{pp}\left(\beta\right)=\left(D\left(A,\beta\right)\right)_{pp}$,
and $\Gamma\left(G,\beta\right)=tr\left(D\left(A,\beta\right)\right).$
Hereafter every time that we write $\Gamma_{pq}$, $\Gamma_{pp}$,
and $\Gamma\left(G\right)$ it should be understood that $\beta\equiv1$. }

\section{\textcolor{black}{Properties of $\Gamma\left(G\right)$}}

\textcolor{black}{In this section we study some of the mathematical
properties of the indices derived from the new matrix function $\left(D\left(A\right)\right).$
In particular, we consider bounds for the double-factorial Estrada
index of graphs. In this section we consider that $\beta\equiv1$,
but the results are trivially extended for any $\beta.$}
\begin{prop}
\textcolor{black}{Let $G$ be a simple connected graph on $n$ nodes.
Then, the double-factorial Estrada index of $G$ is bounded as follows}

\textcolor{black}{
\begin{eqnarray}
\Gamma\left(G\right) & \leq & \frac{1}{2}\left(\sqrt{2\pi}\tanh\left(\frac{k(n-1)}{\sqrt{2}}\right)+2\right)\exp\left(\frac{(n-1)^{2}}{2}\right)\\
 &  & +\frac{(n-1)}{2}\left(\sqrt{2\pi}\tanh\left(\frac{-k}{\sqrt{2}}\right)+2\right)\exp\left(\frac{1}{2}\right),\nonumber 
\end{eqnarray}
with equality if and only if the graph is complete.}
\end{prop}
\begin{proof}
\textcolor{black}{Let $l$ be an edge of $G$ and assume that $G$
is not trivial, i.e., it contains at least one edge. Let $G-l$ be
the graph resulting from removing the edge $l$ from $G$. Let $\mu_{k}\left(G\right)$
be the number of closed walks of length $k$ in $G$. Then, $\mu_{k}\left(G-l\right)=\mu_{k}\left(G\right)-\mu_{k}\left(G:l\right)$,
where $\mu_{k}\left(G:l\right)$ is the number of closed walks of
length $k$ in $G$ which contain the edge $l$. Consequently,}

\textcolor{black}{
\[
\sum_{p=1}^{n}\left(\sum_{k=0}^{\infty}\frac{\mu_{k}\left(G-l\right)}{k!!}\right)_{pp}\leq\sum_{p=1}^{n}\left(\sum_{k=0}^{\infty}\frac{\mu_{k}\left(G\right)}{k!!}\right)_{pp},
\]
which means that $\Gamma\left(G\right)\leq\Gamma\left(K_{n}\right)$
with equality if the graph is the complete graph won $n$ vertices.
We now obtain the formula for $\Gamma\left(K_{n}\right).$ The spectrum
of $K_{n}$ is $\lambda_{1}=n-1$ with multiplicity one and $\lambda_{j\geq2}=-1$
with multiplicity $n-1$ from which the result immediately appears.}
\end{proof}
\begin{cor}
\textcolor{black}{\label{cor: Spanning tree}Let $G$ be a graph and
let $T$ be a spanning tree of $G$. Then}

\textcolor{black}{
\begin{equation}
\Gamma\left(G\right)\geq\Gamma\left(T\right).
\end{equation}
}
\end{cor}
\textcolor{black}{In the next part of this section we will find a
lower bound for the double-factorial Estrada index of graphs. First,
we find an expression for this index for the path graph $P_{n}$,
which will be needed for proving the lower bound.}
\begin{lem}
\textcolor{black}{Let $P_{n}$ be a path with $n$ nodes. Then, when
$n\rightarrow\infty$
\begin{equation}
\Gamma\left(P_{n}\right)=eI_{0}(1)\left(n+\frac{1}{2}\right)-\frac{e^{2}}{2}+o\left(n\right),\label{eq:exact formula path}
\end{equation}
}where $I_{\gamma}\left(z\right)$ are modified Bessel functions of
the first kind \cite{Bessel}.
\end{lem}
\begin{proof}
\textcolor{black}{Let $n$ be number of nodes in $P_{n}$ 
\begin{eqnarray}
\Gamma\left(P_{n}\right) & = & \stackrel[p=1]{n}{\sum}\Gamma_{pp}\left(P_{n}\right).
\end{eqnarray}
}

\textcolor{black}{By substituting the eigenvalues and eigenvectors
of the path graph into the expression for $\Gamma_{pp}$ we obtain
\begin{eqnarray}
{\color{black}\Gamma_{pp}\left(P_{n}\right)} & {\color{black}=} & {\color{black}\frac{2}{n+1}\stackrel[j=1]{n}{\sum}\sin^{2}\left(\frac{j\pi p}{n+1}\right)\exp\left(2\cos^{2}\left(\frac{j\pi}{n+1}\right)\right)}\\
{\color{black}{\color{black}}} & {\color{black}{\color{black}=}} & {\color{black}{\color{black}{\color{black}{\color{black}{\color{black}{\color{black}{\color{black}{\color{red}{\color{black}\frac{{\color{black}1}}{n+1}}{\color{black}\stackrel[j=1]{n}{\sum}}{\color{black}{\color{black}\left[1-\cos\left(\frac{2j\pi p}{n+1}\right)\right]\exp}\left({\color{black}1+\cos\left(\frac{2j\pi}{n+1}\right)}\right)}}}}}}}}}\\
{\color{black}{\color{black}}} & {\color{black}{\color{black}=}} & {\color{black}{\color{black}{\color{black}{\color{black}{\color{black}{\color{red}{\color{black}\frac{e}{n+1}}{\color{black}\stackrel[j=1]{n}{\sum}\left[1-\cos\left(\frac{2j\pi p}{n+1}\right)\right]}{\color{black}\exp\left(\cos\left(\frac{2j\pi}{n+1}\right)\right).}}}}}}}\label{eqb}
\end{eqnarray}
 Now, when $n\rightarrow\infty$ the summation in (\ref{eqb}) can
be evaluated by making use of the following integral 
\begin{equation}
\Gamma_{pp}\left(P_{n}\right)=\frac{e}{\pi}\int_{0}^{\pi}\exp(\cos\theta)d\theta-\frac{e}{\pi}\int_{0}^{\pi}\cos\left(p\theta\right)\exp(\cos\theta)d\theta+o\left(n\right),
\end{equation}
 where $\theta=\frac{2j\pi}{n+1}$. Thus, when $n\rightarrow\infty$
we have 
\begin{equation}
\Gamma_{pp}\left(P_{n}\right)=e\left(I_{0}(1)-I_{p}(1)\right)+o\left(n\right).
\end{equation}
}

\textcolor{black}{We then have}

\textcolor{black}{
\begin{equation}
\Gamma\left(P_{n}\right)=eI_{0}(1)n-e\stackrel[p=1]{n}{\sum}I_{p}(1)+o\left(n\right).\label{eq:path graph}
\end{equation}
}

\textcolor{black}{Replacing the sum $\stackrel[p=1]{\infty}{\sum}I_{p}(x)=\frac{1}{2}\left(e^{x}-I_{0}\left(x\right)\right)$
we finally obtain the result when $n\rightarrow\infty$. }
\end{proof}
\textcolor{black}{Now, we can find the lower bound for the double-factorial
Estrada index of graphs.}
\begin{lem}
\textcolor{black}{Let $G$ be a graph with $n$ nodes. Then,}

\textcolor{black}{
\begin{equation}
\Gamma\left(G\right)\geq eI_{0}(1)\left(n+\frac{1}{2}\right)-\frac{e^{2}}{2},
\end{equation}
with equality if and only if $G$ is the path graph $P_{n}$, }where
$I_{\gamma}\left(z\right)$ are modified Bessel functions of the first
kind \cite{Bessel}.
\end{lem}
\begin{proof}
\textcolor{black}{Let $T$ be a tree with n nodes}

\textcolor{black}{
\begin{eqnarray}
\Gamma\left(T\right) & = & \sum_{j=1}^{n}\exp\left(\lambda_{j}^{2}/2\right)\\
 & = & n+\sum_{j}\lambda_{j}^{2}/2+\sum_{j}\frac{\left(\lambda_{j}^{2}/2\right)^{2}}{2!}+\sum_{j}\frac{\left(\lambda_{j}^{2}/2\right)^{3}}{3!}+\cdots\\
 & \geq & n+m+\dfrac{m^{2}}{2}+\frac{m^{3}}{6}+\frac{m^{4}}{24}=n+\sum_{k=1}^{4}\frac{(n-1)^{k}}{k}=F(n)
\end{eqnarray}
}where $m=n-1$ is the number of edges in the path graph.\textcolor{black}{{}
It is easy to show that for $n\geq3$}

\textcolor{black}{
\begin{equation}
\Gamma\left(T\right)\geq F(n)\geq eI_{0}(1)\left(n+\frac{1}{2}\right)-\frac{e}{2}^{2}=\Gamma\left(P_{n}\right).
\end{equation}
}

\textcolor{black}{Using Corollary \ref{cor: Spanning tree} we easily
see that $\Gamma\left(G\right)\geq\Gamma\left(T\right)\geq\Gamma\left(P_{n}\right)$,
which proves the result.}
\end{proof}
\textcolor{black}{In closing, the double-factorial Estrada index of
graphs is bounded $\Gamma\left(P_{n}\right)\leq\Gamma\left(G\right)\leq\Gamma\left(K_{n}\right)$,
which is similar to $EE\left(G\right)$. In the next section we will
see that the two indices display significant differences when used
to analyze graphs containing significantly large chordless cycles
or holes.}

\section{\textcolor{black}{Graphs with holes}}

\textcolor{black}{We now consider a graph $G$ and and two nodes $p$
and $q$ in $G$. Suppose that all the number of subgraphs of sizes
smaller than a certain value $k_{0}$ to which the node $p$ belongs
to is larger than that for the node $q$. Also consider that the node
$q$ is involved in a larger number of subgraphs of size larger than
$k_{0}$ than the node $p$. This situation can be found in any graph
containing holes. A hole is a chordless cycle, that is a closed sequence
of nodes in $G$ such that each two adjacent nodes in the sequence
are connected by an edge and each two non-adjacent nodes in the sequence
are not connected by any edge in $G$. Then, the situation previously
described can appear when one of the nodes is in a chordless cycle
and the other not. An example of this situation is represented in
Figure \ref{Chordless cycle}. Here node $A$ takes place, for instance
in 3 triangles, while node $B$ takes place in 6. However, the number
of walks of length larger than 17 is bigger for node $A$, which indeed
is part of a chordless cycle of length 18, than for the node $B$. }

\begin{figure}
\begin{centering}
\textcolor{black}{\includegraphics[width=1\textwidth]{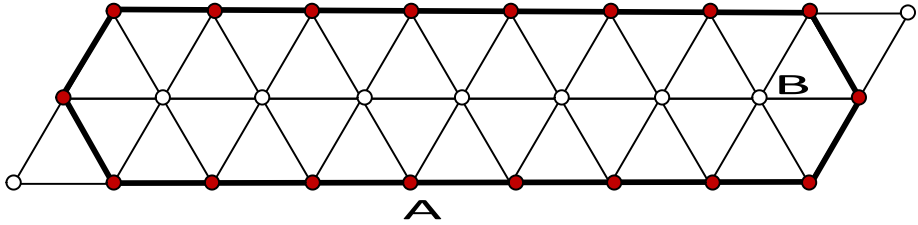}}
\par\end{centering}
\caption{Illustration of a triangular lattice with 27 nodes. The nodes marked
in red belong to a chordless cycle of length 18, which is\textcolor{black}{{}
highlighted with bolded edges. }The nodes $A$ and $B$ are discussed
in the main text.}

\label{Chordless cycle}
\end{figure}

\textcolor{black}{Now, let us express mathematically the situation
that we have described in the precedent paragraph. That is, $\left(A^{k}\right)_{pp}>\left(A^{k}\right)_{qq}$
for all $k<k_{0}$, and $\left(A^{k}\right)_{pp}<\left(A^{k}\right)_{qq}$
for all $k>k_{0}$, where $k_{0}\gg1$. Let us now consider the difference
between the double factorial subgraph centrality and the single-factorial
version of it for the node $p$,}

\textcolor{black}{
\begin{eqnarray}
\varDelta_{p} & = & \Gamma_{pp}\left(G\right)-EE_{pp}\left(G\right)\\
 & = & \dfrac{1}{6}\left(A^{3}\right)_{pp}+\dfrac{1}{12}\left(A^{4}\right)_{pp}+\dfrac{7}{120}\left(A^{5}\right)_{pp}+\dfrac{7}{360}\left(A^{6}\right)_{pp}+\cdots\\
 & = & \sum_{k=3}^{\infty}\dfrac{\left(k-1\right)!!-1}{k!!\left(k-1\right)!!}\left(A^{k}\right)_{pp}.
\end{eqnarray}
}

\textcolor{black}{Then, in the situation previously described it is
plausible that the functions $\varDelta_{p}$ and $\varDelta_{q}$
follow similar trends to the spectral moments of the adjacency matrix.
That is, $\varDelta_{p}>\varDelta_{q}$ for all $k<k_{0}$, and $\varDelta_{p}<\varDelta_{q}$
for all $k>k_{0}$.}

\textcolor{black}{For the example illustrated in Figure \ref{Chordless cycle}
we give in Figure \ref{Moments} the plots of $\left(A^{k}\right)_{pp}/k!!$
and $\left(A^{k}\right)_{pp}/k!$ for the nodes labelled as $A$ and
$B$ in Figure \ref{Chordless cycle}, as well as the plot of $\varDelta_{A}$
and $\varDelta_{B}.$ As can be seen the node $A$, which is in the
chordless cycle, has smaller contribution from small subgraphs than
node $B$, which is outside the hole. However, node $A$ takes place
in longer subgraphs, such as its own chordless cycle, than node $B$
and consequently $\varDelta_{A}>\varDelta_{B}$ for $k>18$. }

\begin{figure}

\begin{centering}
\textcolor{black}{\includegraphics[width=0.33\textwidth]{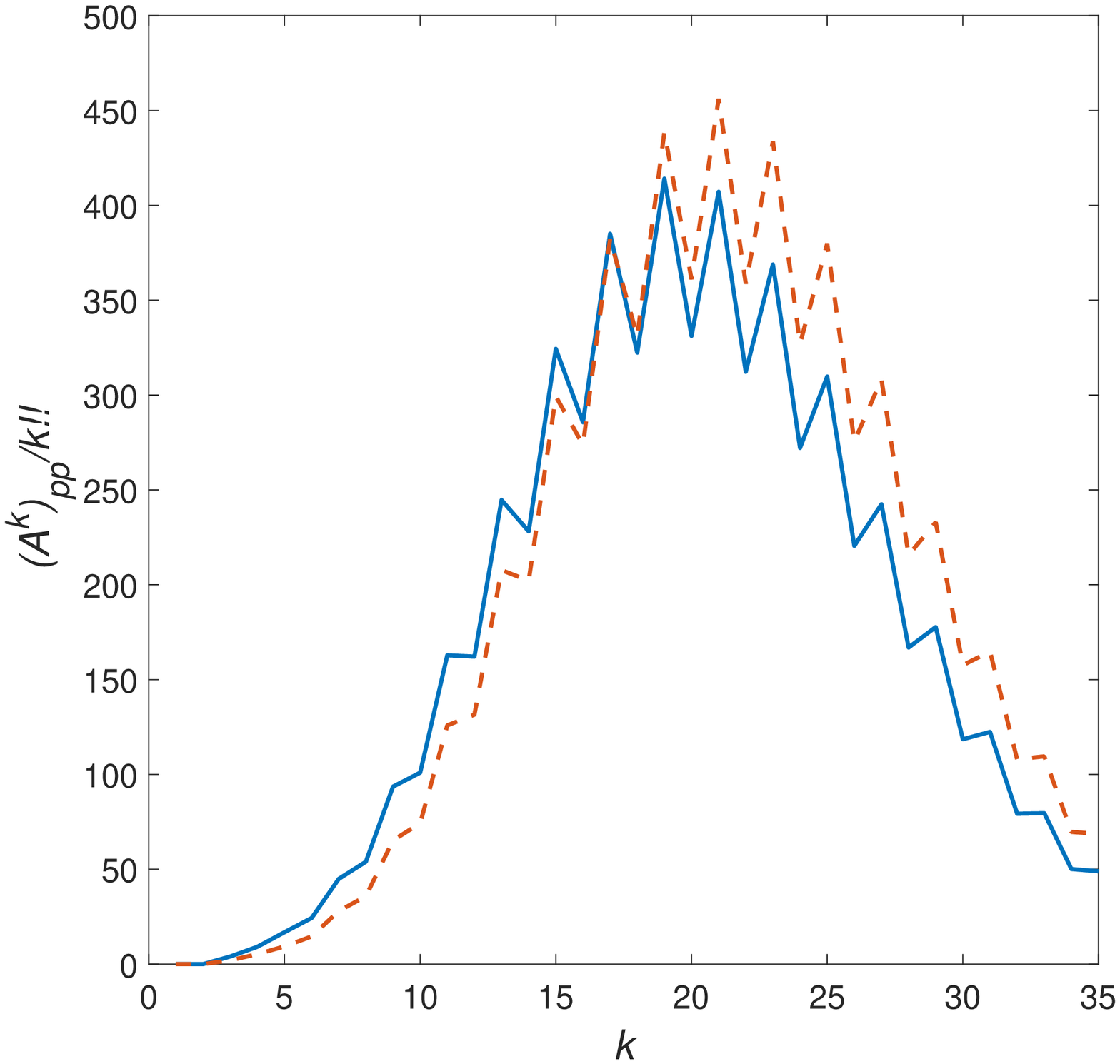}\includegraphics[width=0.33\textwidth]{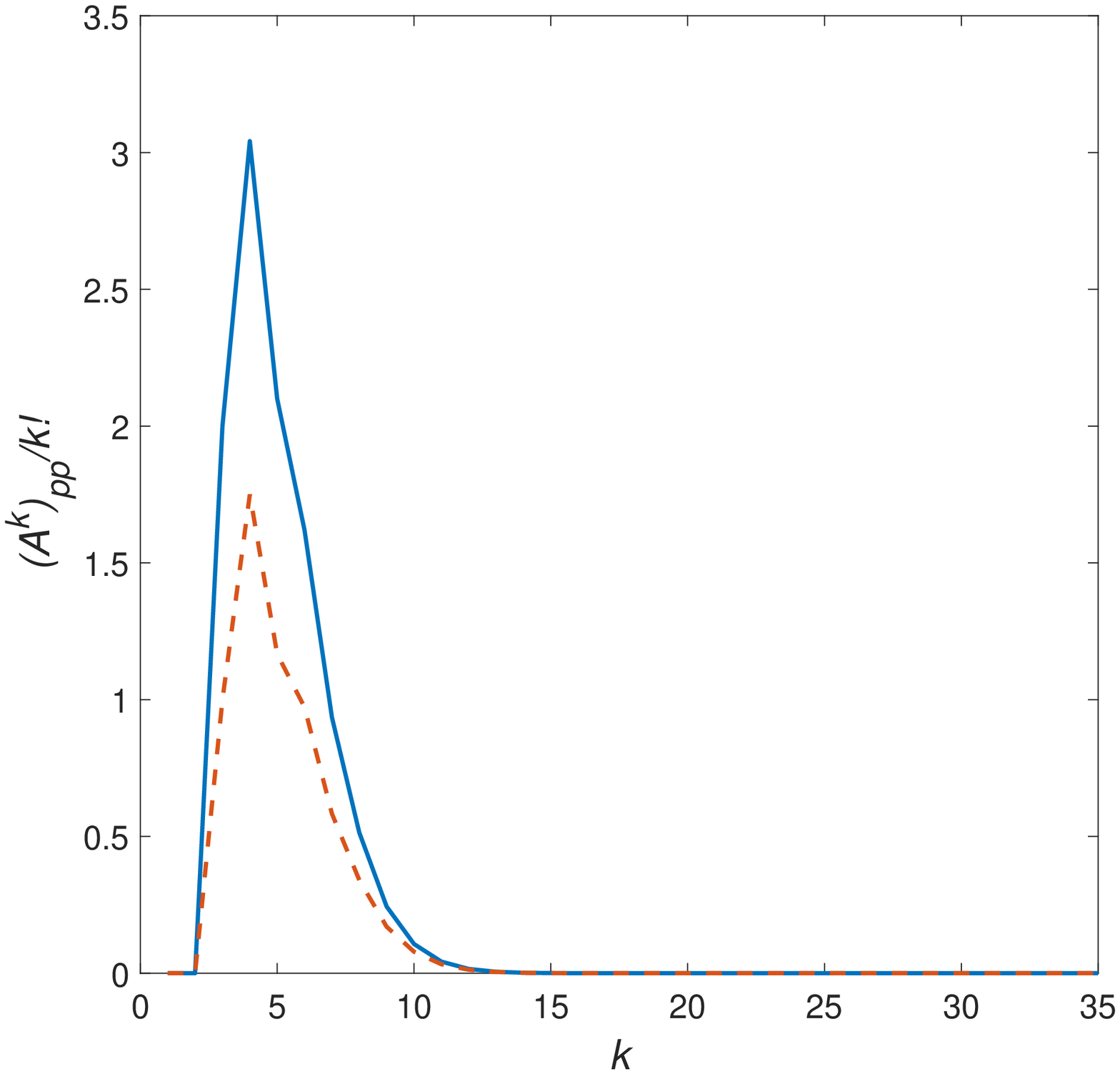}\includegraphics[width=0.33\textwidth]{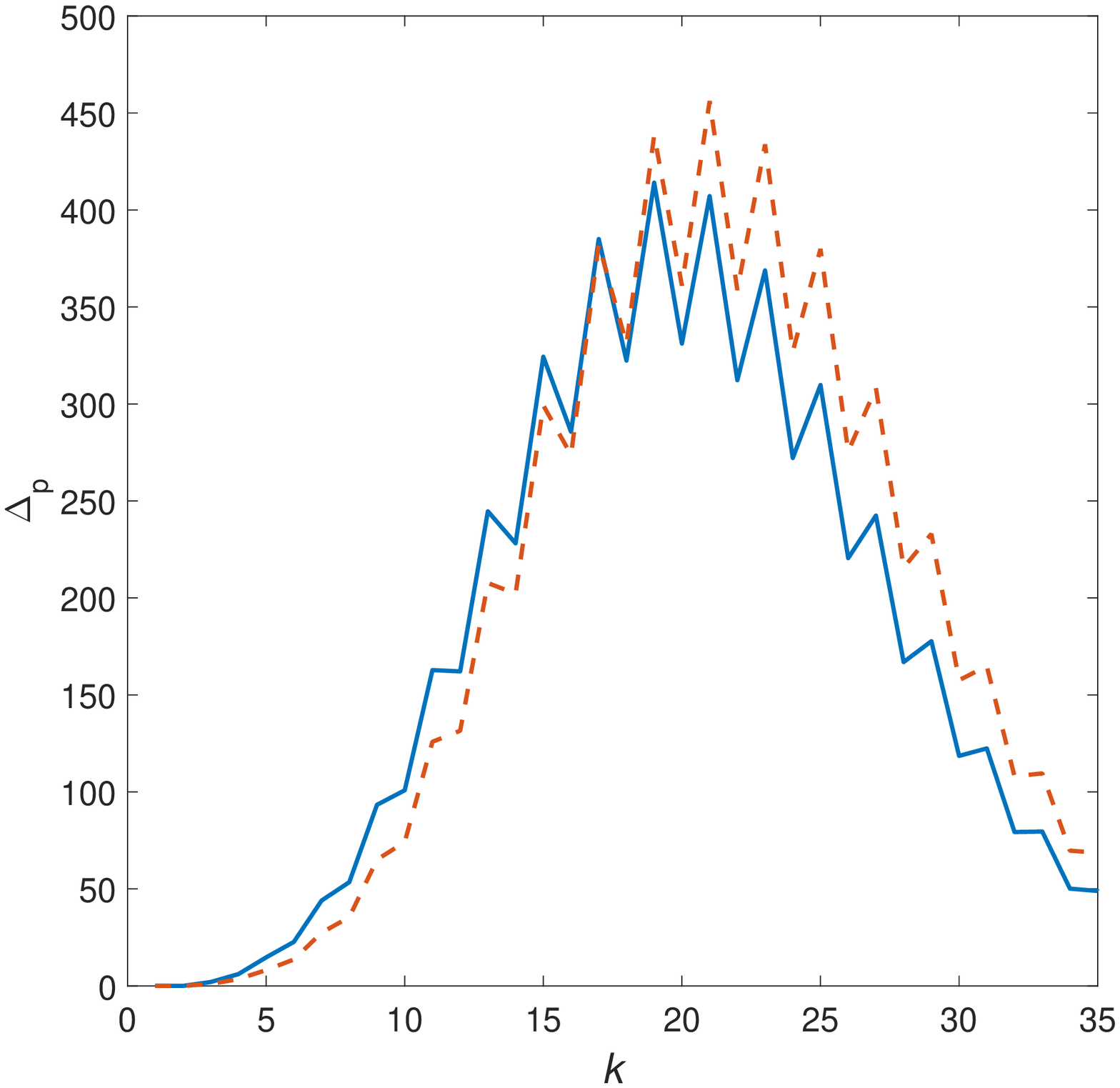}}
\par\end{centering}
\caption{Values of $c_{k}\left(A^{k}\right)_{pp}$ for $c_{k}=1/k!$ and $c_{k}=1/k!$!
for different values of $k$ in the triangular lattice illustrated
in Figure \ref{Chordless cycle} for two nodes. Blue (continuous)
line is for the node $B$ and red (broken) line is for the node $A$.
The panel on the right illustrates the values of the differences between
both types of penalizations for the two studied nodes (see text for
details).}

\label{Moments}
\end{figure}

\textcolor{black}{The consequences of the previous kind of situation
is that there is a different ranking of the nodes according to the
subgraph centrality (or communicability) based on the single and double-factorial
penalization. For instance, according to the single factorial penalization
$G_{AA}\approx9.1134$ and $G_{BB}\approx14.6272$. That is, the node
$B$ is more central than node $A$.. However, according to the double-factorial
penalization we obtain $\varGamma_{AA}\approx3038.6$ and $\Gamma_{BB}\approx2806.8$,
which indicates that indeed the node $A$ is more central than node
$B$. In many cases this difference in ranking is observed for many
pairs of nodes in a network, which produces a lack of correlation
between the corresponding parameters. In Figure \ref{Triangular correlations}
we illustrate the correlations between the subgraph centralities (left
panel) and communicability (right panel) for the nodes and pairs of
nodes, respectively, in the network illustrated in Figure \ref{Chordless cycle}.}

\textcolor{black}{}
\begin{figure}

\begin{centering}
\textcolor{black}{\includegraphics[width=0.45\textwidth]{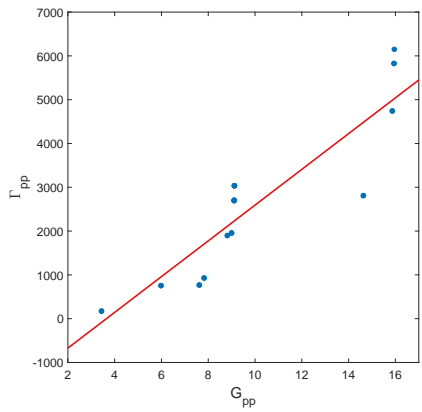}}
\par\end{centering}
\textcolor{black}{\caption{Scatterplot of the subgraph centrality (left panel) and communicability
(right panel) based on the single and double\textcolor{red}{-}\textcolor{black}{factorial
penalization for} all the nodes and pairs of nodes, respectively,
in the graph illustrated in Figure \ref{Triangular correlations}.}
}

\textcolor{black}{\label{Triangular correlations}}
\end{figure}

\textcolor{black}{The graph previously considered is a triangular
lattice, which is a planar graph. The situation previously considered
where a chordless cycle appears can be seen frequently in these type
of graphs and it is consequently of importance for studying certain
kinds of real-world networks as we will see in the next section. However,
such kinds of examples are not exclusive to planar graphs. As a simple
illustration we destroy the planarity of the graph in Figure \ref{Triangular correlations}
by adding a few edges but keeping the same chordless cycle as in the
original triangular lattice. As can be seen in Figure \ref{Nonplanar Lattice}
there is a total lack of correlation between the communicability obtained
with the single and double factorial for all pairs of nodes in this
graph. We will show more realistic examples from real-world networks
in the next section of the paper.}

\begin{figure}

\begin{centering}
\textcolor{black}{\includegraphics[width=0.75\textwidth]{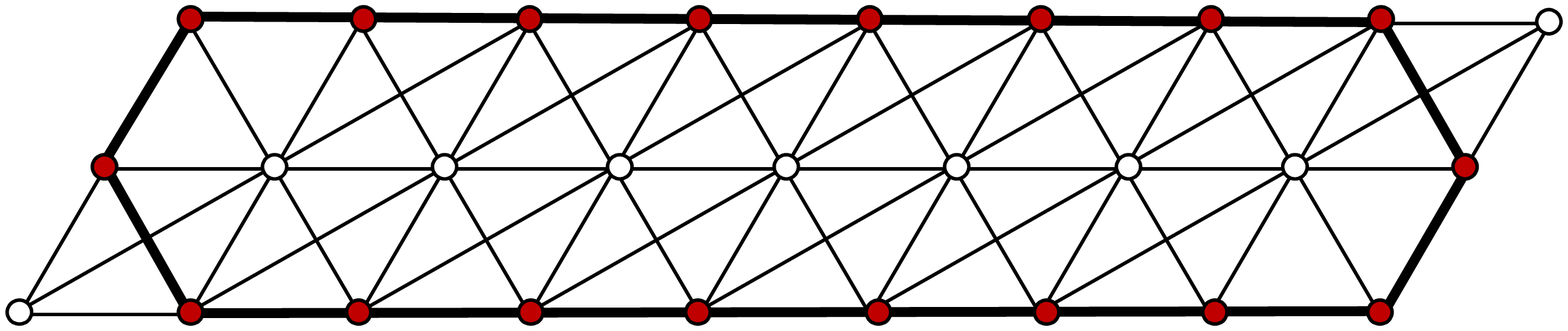}}
\par\end{centering}
\begin{centering}
\textcolor{black}{\includegraphics[width=0.55\textwidth]{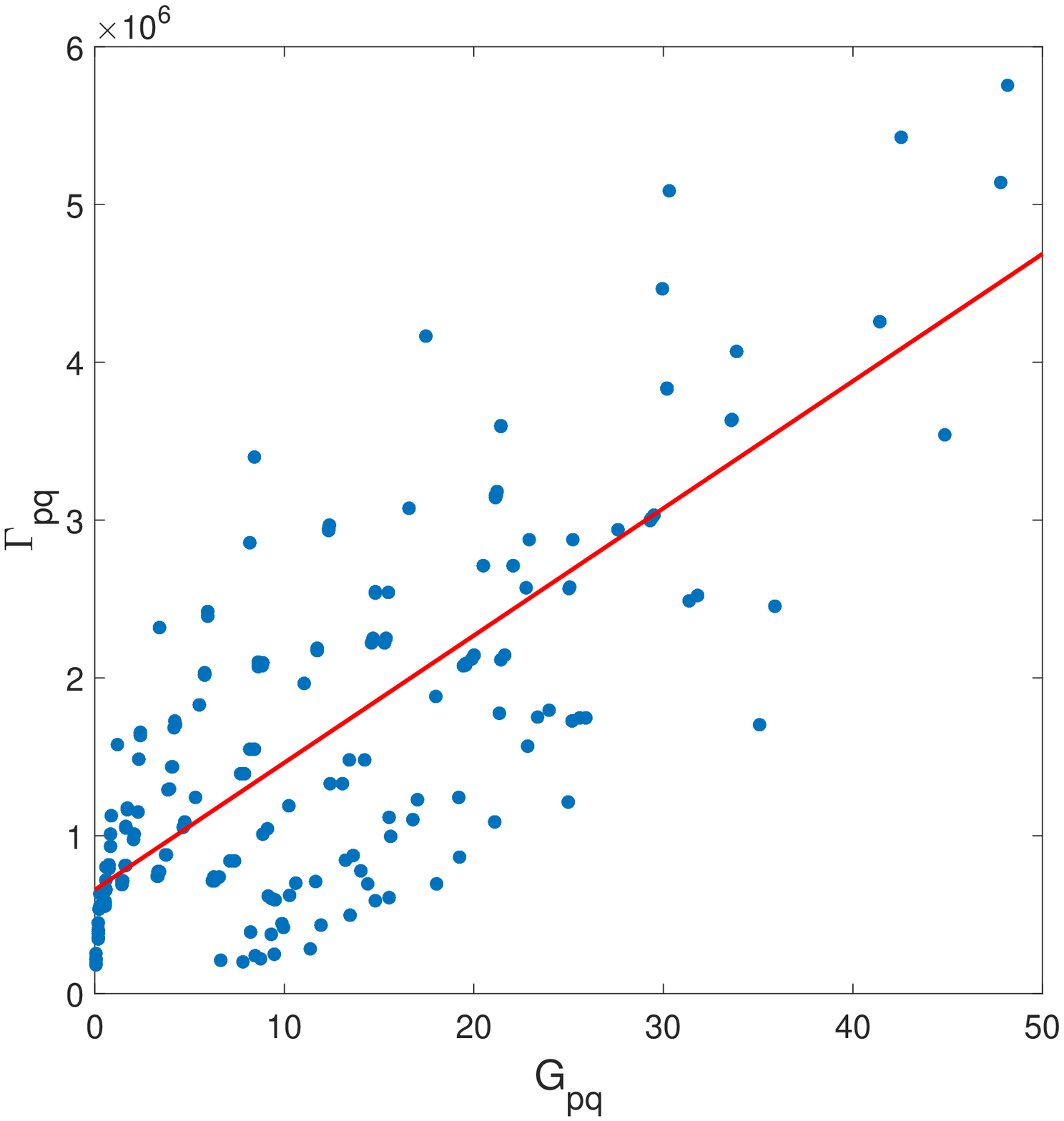}}
\par\end{centering}
\caption{(Top) Nonplanar graph obtained from the triangular lattice with 27
nodes illustrated in \ref{Triangular correlations}. The nodes marked
in red belong to a chordless cycle of length 18,\textcolor{red}{{} }\textcolor{black}{which
is highlighted with bolded edges. (}bottom) Scatterplot of the communicability
based on the single and double factorial for all pairs of nodes for
the graph in the top panel.}

\label{Nonplanar Lattice}
\end{figure}

\section{\textcolor{black}{Analysis of Real-World Networks }}

\subsection{\textcolor{black}{General analysis}}

\textcolor{black}{An important problem to be considered in practical
applications is that the entries of $\frac{A^{k}}{k!!}$ grow very
fast with $k$ in large graphs with relatively high density. Although
most real-world networks are sparse, the calculation of indices based
on $D\left(A\right)$ can be affected by the presence of these very
large numbers. For instance, for a network representing the synaptic
connections among the neurons of the worm }\textit{\textcolor{black}{C.
elegans}}\textcolor{black}{, which has $n=280$ nodes and edge density
$d=0.0505$ the entries of $D\left(A\right)$ are bigger than $10^{110}$,
which far exceeds the largest finite floating-point number in IEEE
single precision ($10^{38}$), but is still below the largest finite
floating-point number in IEEE double precision ($10^{308}$). However,
for the network representing the USA system of airports having $n=332$
nodes and $d=0.0387$ the entries of $D\left(A\right)$ exceed this
maximum floating number and a program like Matlab\textregistered{}
returns infinity for all its entries. In those cases the adjacency
matrix can be multiplied by $\beta<1$ in order to reduce the magnitude
of the entries of $D\left(A\right)$ as we will illustrate in some
of the examples in this section. We then study the influence of this
parameter $\beta$ on the function $D\left(A\right)$.}

\textcolor{black}{In this subsection we study networks that do not
contain significantly large chordless cycles. We now conduct a computational
study of the index $\Gamma\left(G,\beta\right)$ of all connected
graphs with $n=4,5,6,7,8$ nodes and compare it with the index $EE\left(G,\beta\right)$
for values $0<\beta\leq1$. In Figure \ref{Estrada indices with 8 nodes}(a)
we illustrate the correlation between the two indices for $\beta=1$
that show the existence of a power-law relation between them. However,
by zooming into the smallest valued region of the indices\textemdash this
region corresponds to graphs with relatively low density of edges\textemdash it
is revealed that such a correlation between the two indices is far
from being simple (see (\ref{Estrada indices with 8 nodes}(b)). This
reveals the fact that decreasing the penalization of the walks in
graphs from the factorial to the double-factorial make non-trivial
changes in the ordering of the graphs, particularly for graphs with
relatively low density of edges. This is very important as most real-world
networks are sparse and we should expect significant differences between
the two different indices for them. More interestingly, we plot the
two indices for $\beta=0.1$ in (\ref{Estrada indices with 8 nodes}(c))
where it can be observed that the correlation between the two indices
have now been dramatically decreased.}

\textcolor{black}{}
\begin{figure}

\begin{centering}
\textcolor{black}{\includegraphics[width=0.33\textwidth]{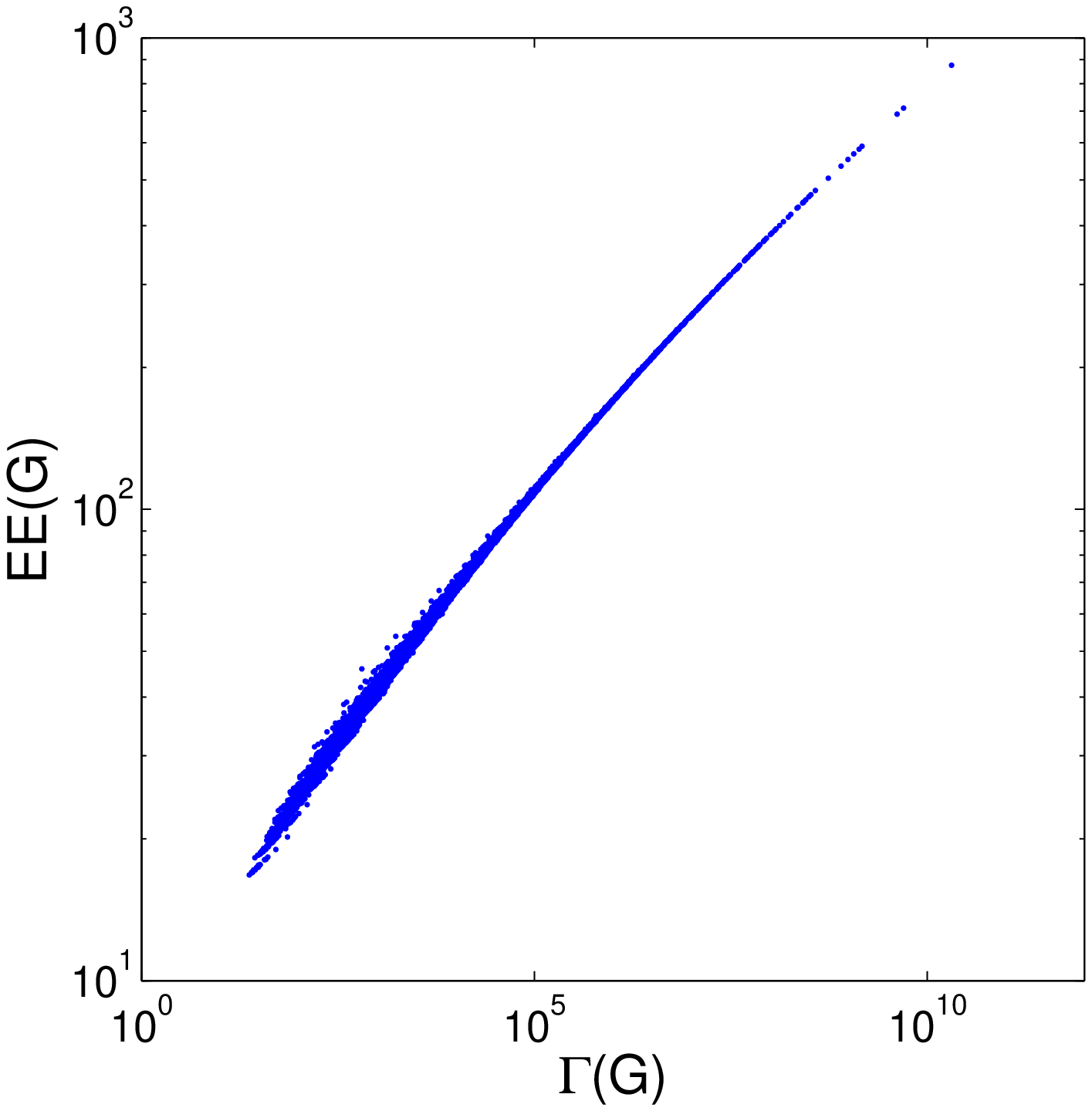}\includegraphics[width=0.33\textwidth]{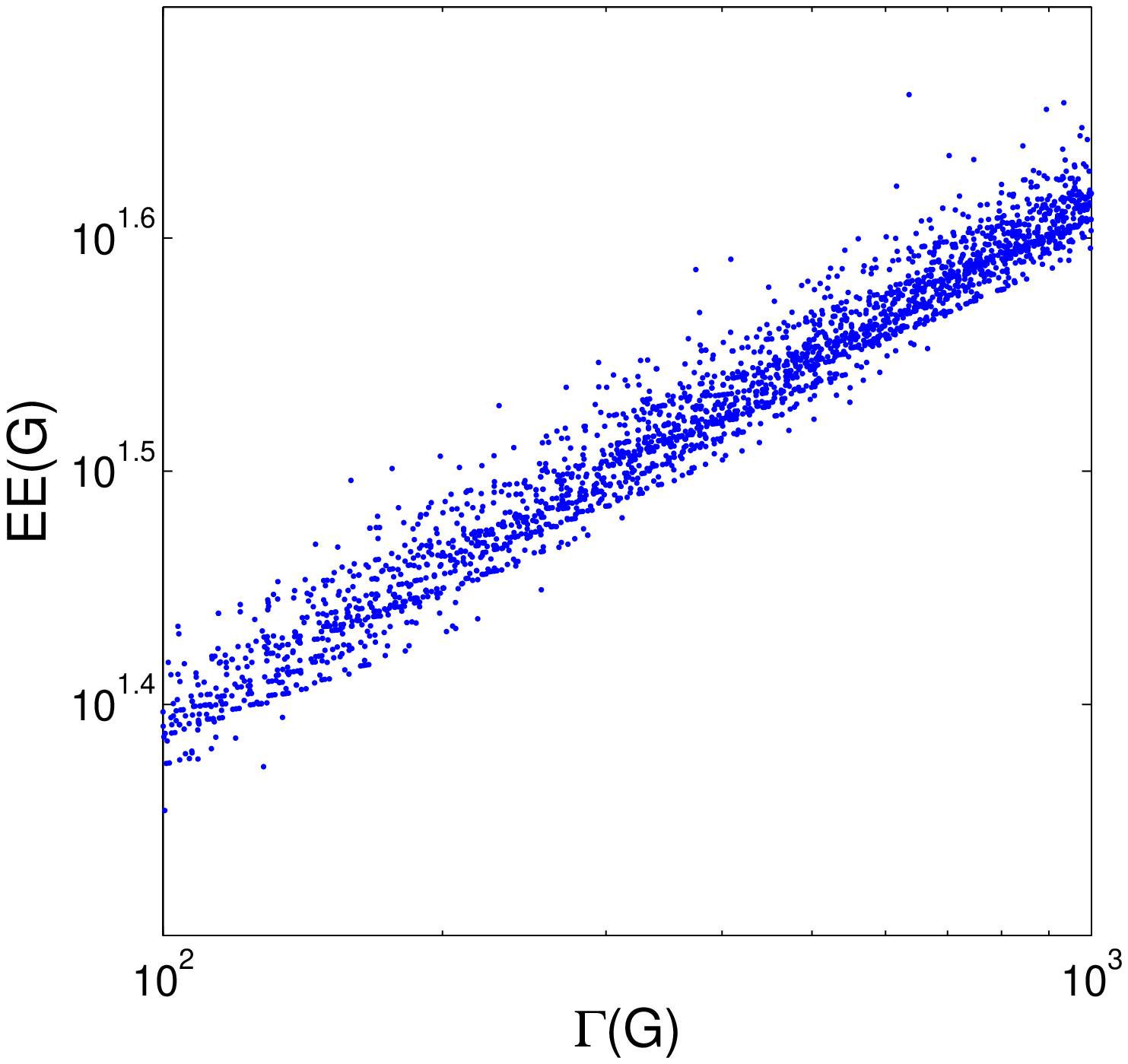}\includegraphics[width=0.33\textwidth]{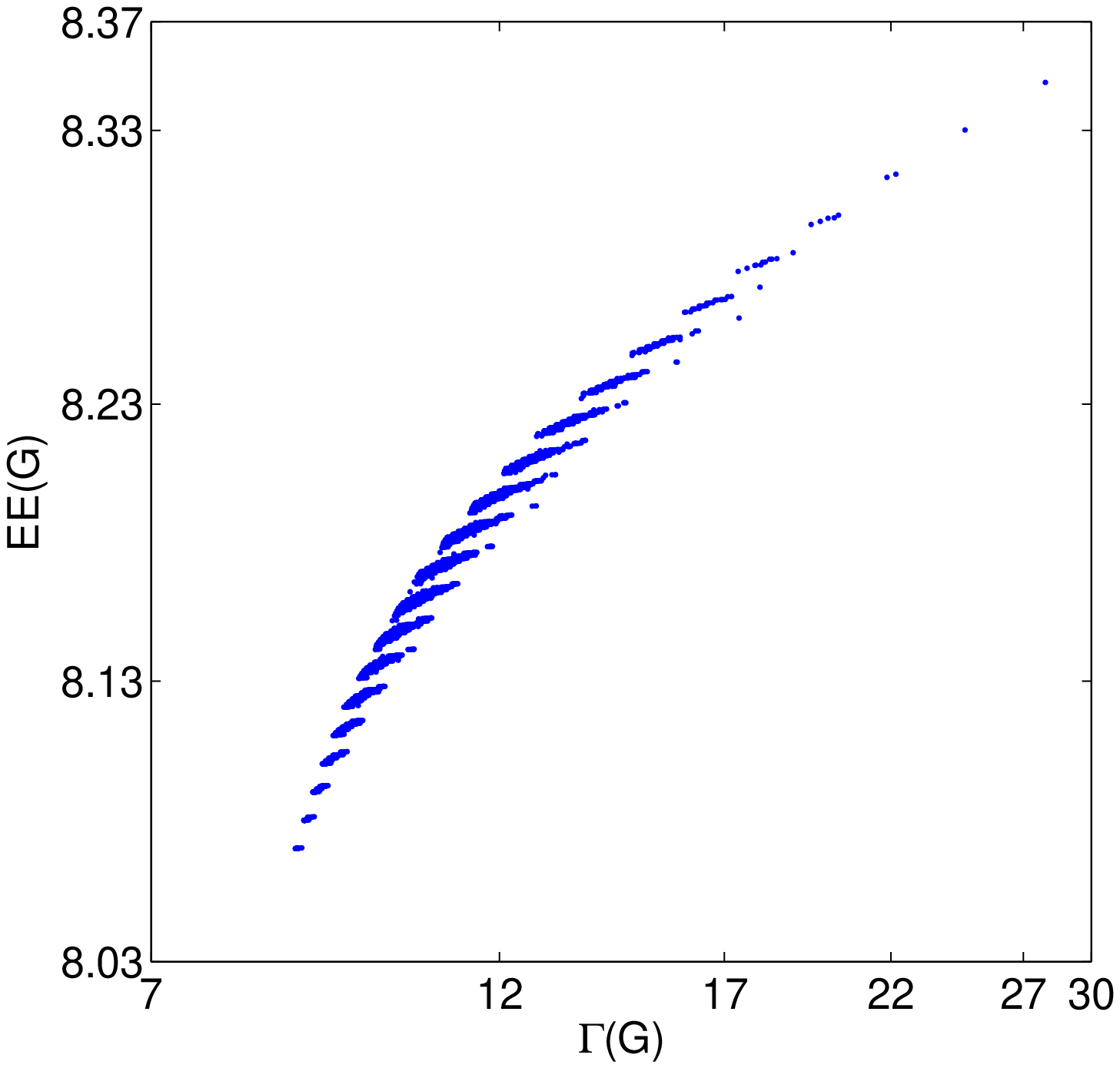}}
\par\end{centering}
\textcolor{black}{\caption{(a) Scatterplot of the indices $EE\left(G,\beta\right)$ and $\Gamma\left(G,\beta\right)$
for all the 11,117 connected graphs on 8 nodes using $\beta=1$. (b)
\textcolor{black}{Magnified plot of the region}\textcolor{red}{{} }$100\leq\Gamma\left(G\right)\leq1000$
for the same plot as in (a). (c) The same as in (a) but using $\beta=0.1$.}
}

\textcolor{black}{\label{Estrada indices with 8 nodes}}
\end{figure}

\textcolor{black}{In order to understand this decay in the correlation
between the two indices we express them in terms of the eigenvalues
of the adjacency matrix:}

\textcolor{black}{
\begin{equation}
EE\left(G,\beta\right)=\sum_{j=1}^{n}\textrm{exp}\left(\beta\lambda_{j}\right),
\end{equation}
}

\textcolor{black}{
\begin{equation}
\Gamma\left(G,\beta\right)=\sum_{j=1}^{n}\textrm{exp}\left(\frac{\beta^{2}\lambda_{j}^{2}}{2}\right)+\sqrt{\dfrac{\pi}{2}}\sum_{j=1}^{n}\tanh\left(\frac{k\beta\lambda_{j}}{\sqrt{2}}\right)\textrm{exp}\left(\frac{\beta^{2}\lambda_{j}^{2}}{2}\right).\label{eq:gamma eigenvalues}
\end{equation}
}

\textcolor{black}{It is easy to see that when $\beta\rightarrow\infty$
both indices are dominated by the principal eigenvalue (spectral radius)
of the adjacency matrix, i.e.,}

\textcolor{black}{
\begin{equation}
EE\left(G,\beta\rightarrow\infty\right)=\textrm{exp}\left(\beta\lambda_{1}\right),
\end{equation}
}

\textcolor{black}{
\begin{eqnarray}
\Gamma\left(G,\beta\rightarrow\infty\right) & = & \textrm{exp}\left(\frac{\beta^{2}\lambda_{1}^{2}}{2}\right)+\sqrt{\dfrac{\pi}{2}}\tanh\left(\frac{k\beta\lambda_{1}}{\sqrt{2}}\right)\textrm{exp}\left(\frac{\beta^{2}\lambda_{1}^{2}}{2}\right)\\
 & \simeq & \left(1+\sqrt{\frac{\pi}{2}}\right)\textrm{exp}\left(\frac{\beta^{2}\lambda_{1}^{2}}{2}\right),
\end{eqnarray}
due to the fact that $\tanh\left(x\right)\approx1$ for $x>5$. Then,
it is evident that both indices are highly correlated. On the contrary,
when $\beta\rightarrow0$, the second term of (\ref{eq:gamma eigenvalues})
becomes more relevant. First of all, in this case the term $\tanh\left(x\right)$
is smaller than one for many of the eigenvalues of the network, which
means that a larger number of eigenvalues and not only those close
to zero make a contribution to this part of the function. Although
the first term of (\ref{eq:gamma eigenvalues}) may still correlate
with $EE\left(G,\beta\right)$, the second term does not, which results
in a lack of global correlation between $\Gamma\left(G,\beta\right)$
and $EE\left(G,\beta\right)$ when $\beta\rightarrow0$. This lack
of correlation for small values of $\beta$ will be useful for the
study of some of the indices derived from the matrix function $D\left(A\right)$.}

\textcolor{black}{We now study a group of 61 real-world networks representing
social, biological, ecological, infrastructural and technological
systems. We first illustrate the correlation between $\Gamma\left(G,\beta\right)$
and $EE\left(G,\beta\right)$ when $\beta=0.2$. To avoid size effects
we normalized both indices by dividing their logarithms by the number
of nodes. As can be seen in Figure \ref{real world networks} in general
there is a good correlation between the two indices except for a few
networks\textendash about one third of the total number of networks
studied\textemdash which display large deviations from the linear
trend observed. That is, there are 19 networks for which $\left(\log\Gamma\left(G,\beta=0.2\right)\right)/n$
is significantly larger than expected from the linear correlation
between this index and $\left(\log EE\left(G,\beta=0.2\right)\right)/n$.
Excluding these 19 outliers the Pearson correlation coefficient between
the two indices is $0.999$. We have calculated the average Watts-Strogatz
clustering coefficient and the global transitivity of all the network
studied. They have average values for the 61 networks studied of $0.259$
and $0.203$, respectively. However, if we consider the networks that
deviate significantly from the linear correlation between the two
indices, the clustering coefficients have average values of $0.415$
and $0.337$, respectively, which are much higher than the average
observed for the total networks studies. Indeed, if we only consider
those networks for which there is a perfect fit between the two indices
studied we obtain average clustering coefficients of $0.187$ and
$0.140$, which confirms that the `anomalous' behavior is observed
for networks with the highest clustering coefficients among all the
networks studied.}

\textcolor{black}{If we consider the difference between the two indices
studied we have}

\textcolor{black}{
\begin{equation}
\Gamma\left(G,\beta\right)-EE\left(G,\beta\right)=\dfrac{\beta^{3}}{6}tr\left(A^{3}\right)+\dfrac{\beta^{4}}{12}tr\left(A^{4}\right)+\dfrac{7\beta^{5}}{120}tr\left(A^{5}\right)+\cdots,
\end{equation}
which clearly indicates that the first term is the one having the
largest contribution. We recall that $t=\dfrac{1}{6}tr\left(A^{3}\right),$
where $t$ is the number of triangles. Then, when $\beta\rightarrow0$
the number of triangles has the largest influence in the difference
between the two indices. Consequently, those networks having the largest
clustering\textemdash which account for the relative abundance of
triangles\textemdash display the largest difference between the two
indices among all the networks studied.}

\textcolor{black}{}
\begin{figure}

\begin{centering}
\textcolor{black}{\includegraphics[width=0.75\textwidth]{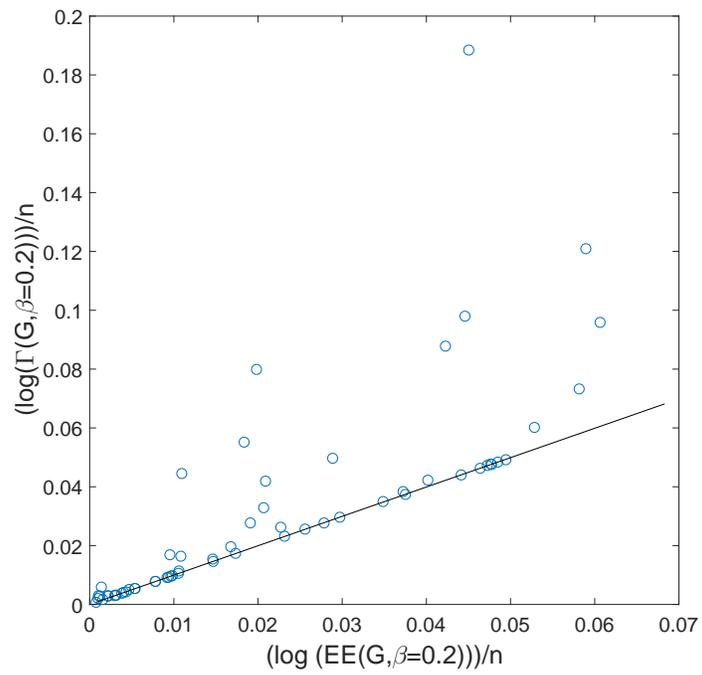}}
\par\end{centering}
\textcolor{black}{\caption{Correlation between the logarithms of $\Gamma\left(G,\beta=0.2\right)$
and $EE\left(G,\beta=0.2\right)$ normalized by the number of nodes
for 61 real-world networks.}
}

\textcolor{black}{\label{real world networks}}
\end{figure}

\subsection{\textcolor{black}{Centrality}}

\textcolor{black}{One of the most important uses of matrix functions
in the study of networks is the definition of centrality indices.
The double-factorial subgraph centrality is defined as the diagonal
entries of the matrix function $D\left(A\right)$, which can be expressed
in terms of the eigenvalues and eigenvectors of the adjacency matrix
as}

\textcolor{black}{
\begin{eqnarray}
\Gamma_{pp}\left(G,\beta\right) & =\sum_{j=1}^{n}\psi_{j,p}^{2}\left[\sqrt{\dfrac{\pi}{2}}\tanh\left(\frac{k\beta\lambda_{j}}{\sqrt{2}}\right)+1\right]\textrm{exp}\left(\frac{\beta^{2}\lambda_{j}^{2}}{2}\right).
\end{eqnarray}
}

\textcolor{black}{The way the double-factorial Estrada index is correlated
to $EE\left(G,\beta\right)$ for large values of $\beta$ is similar
to how the double-factorial subgraph centrality $\Gamma_{pp}\left(G,\beta\right)$
is also correlated to the subgraph centrality $G_{pp}\left(G,\beta\right)$
for large values of $\beta$. In Figure \ref{subgraph centralities}
we illustrate the correlations between the subgraph centralities $G_{pp}\left(G,\beta\right)$
and $\Gamma_{pp}\left(G,\beta\right)$ for the protein-protein interaction
network of yeast (top plots) and the network of directors in the corporate
elite in US (bottom plots). As can be seen in the plots on the left
hand side of the Figure, for $\beta=1$ there is a very good linear
relation between both centralities as expected from the fact that
they can both be approximated by}

\textcolor{black}{
\begin{equation}
G_{pp}\left(G,\beta\rightarrow\infty\right)=\psi_{1,p}^{2}\textrm{exp}\left(\beta\lambda_{1}\right),
\end{equation}
}

\textcolor{black}{
\begin{eqnarray}
\Gamma_{pp}\left(G,\beta\rightarrow\infty\right) & = & \psi_{1,p}^{2}\exp\left(\frac{\beta^{2}\lambda_{1}^{2}}{2}\right)\\
 & + & \textrm{\ensuremath{\sqrt{\dfrac{\pi}{2}}\tanh\left(\frac{k\beta\lambda_{1}}{\sqrt{2}}\right)\textrm{exp}\left(\frac{\beta^{2}\lambda_{1}^{2}}{2}\right)}}\\
 & \simeq & \left(1+\sqrt{\frac{\pi}{2}}\right)\psi_{1,p}^{2}\textrm{exp}\left(\frac{\beta^{2}\lambda_{1}^{2}}{2}\right).
\end{eqnarray}
}

\textcolor{black}{However, when $\beta\rightarrow0$, as in the right
hand side plots of Figure \ref{subgraph centralities}, this correlation
disappears and both indices differ significantly for several nodes
in these networks. The reason for this difference is analogous to
the one explained in the previous section for the corresponding Estrada
indices.}

\textcolor{black}{}
\begin{figure}
\begin{centering}
\textcolor{black}{\includegraphics[width=0.5\textwidth]{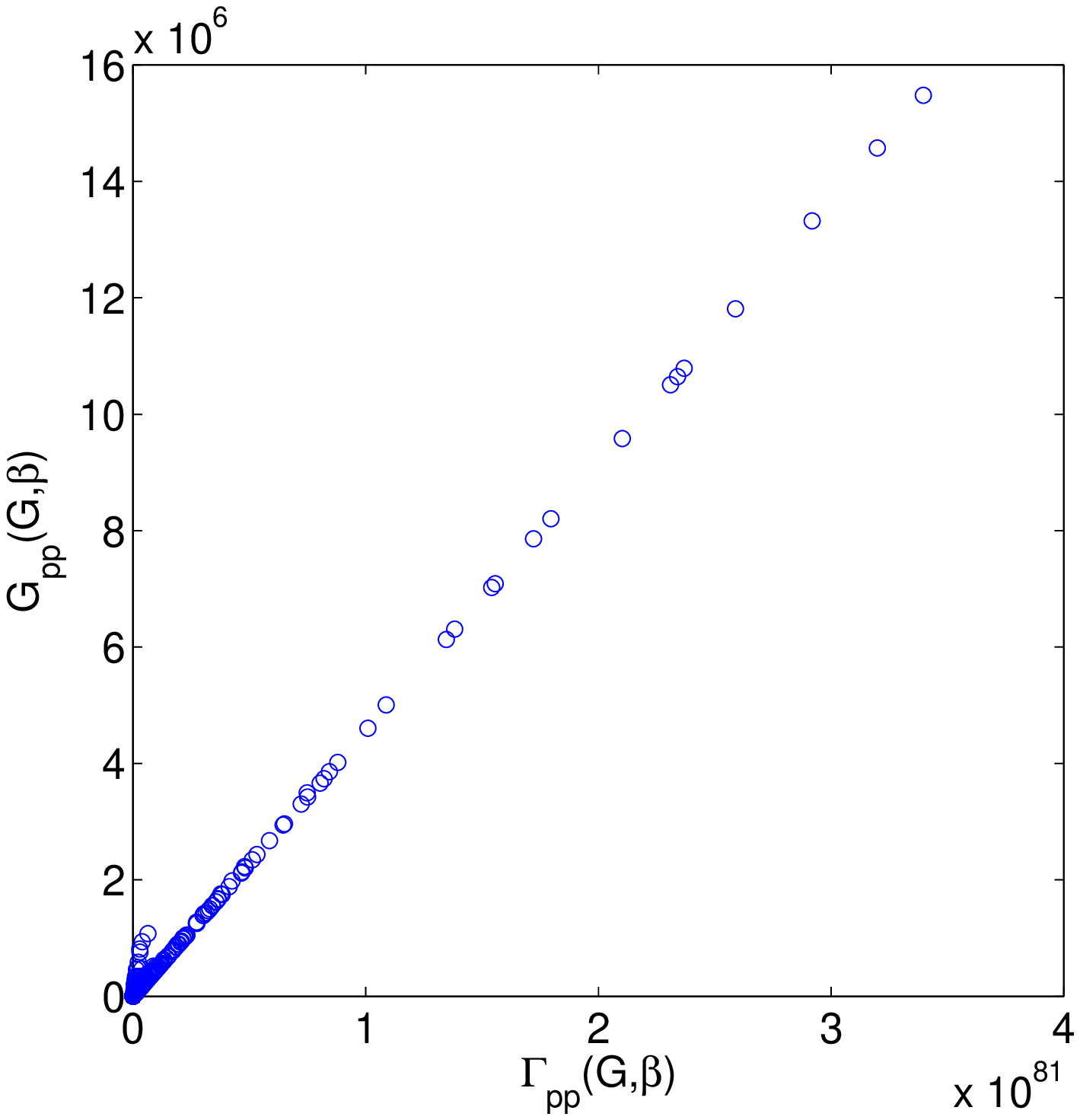}\includegraphics[width=0.5\textwidth]{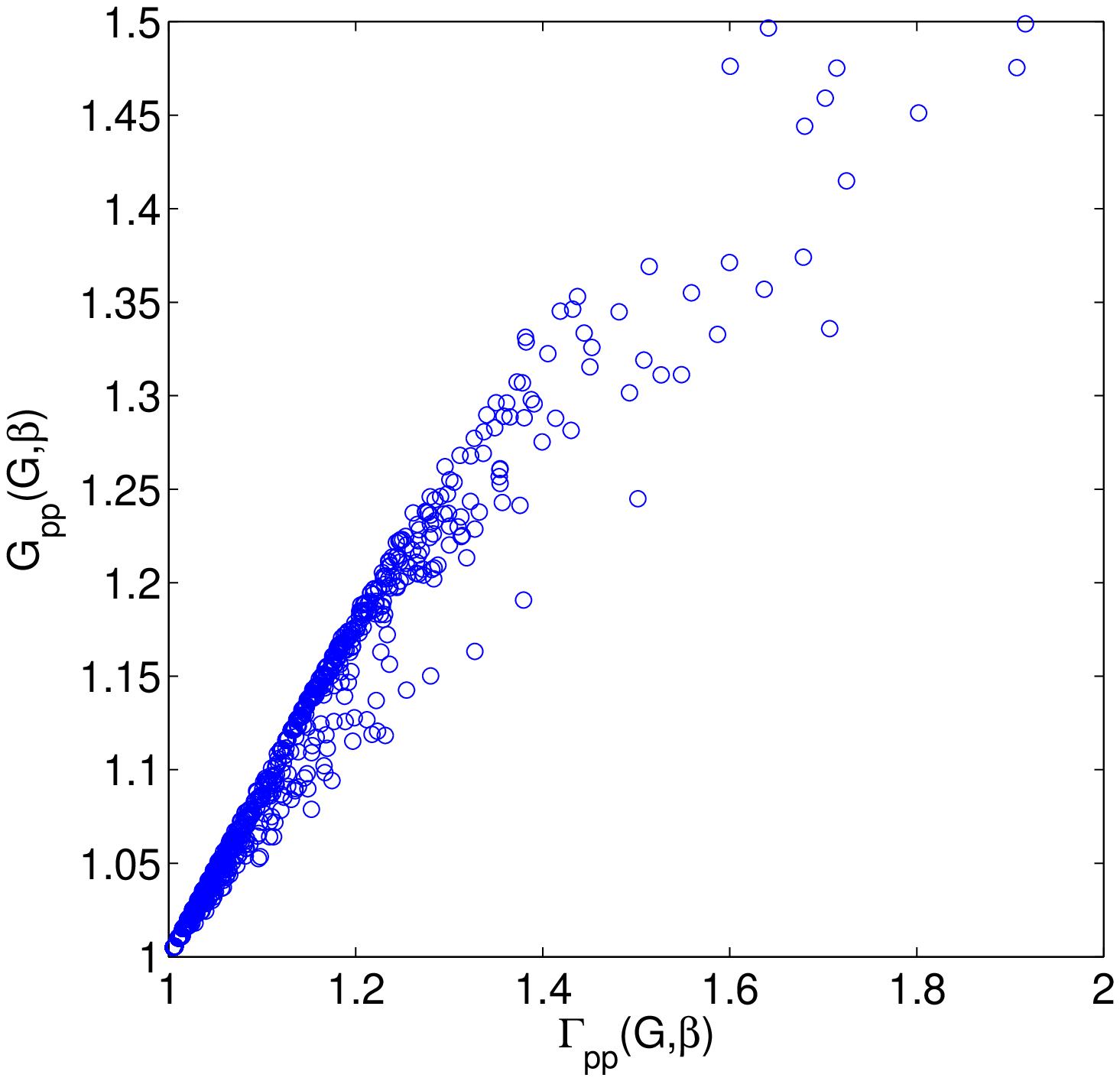}}
\par\end{centering}
\begin{centering}
\textcolor{black}{\includegraphics[width=0.5\textwidth]{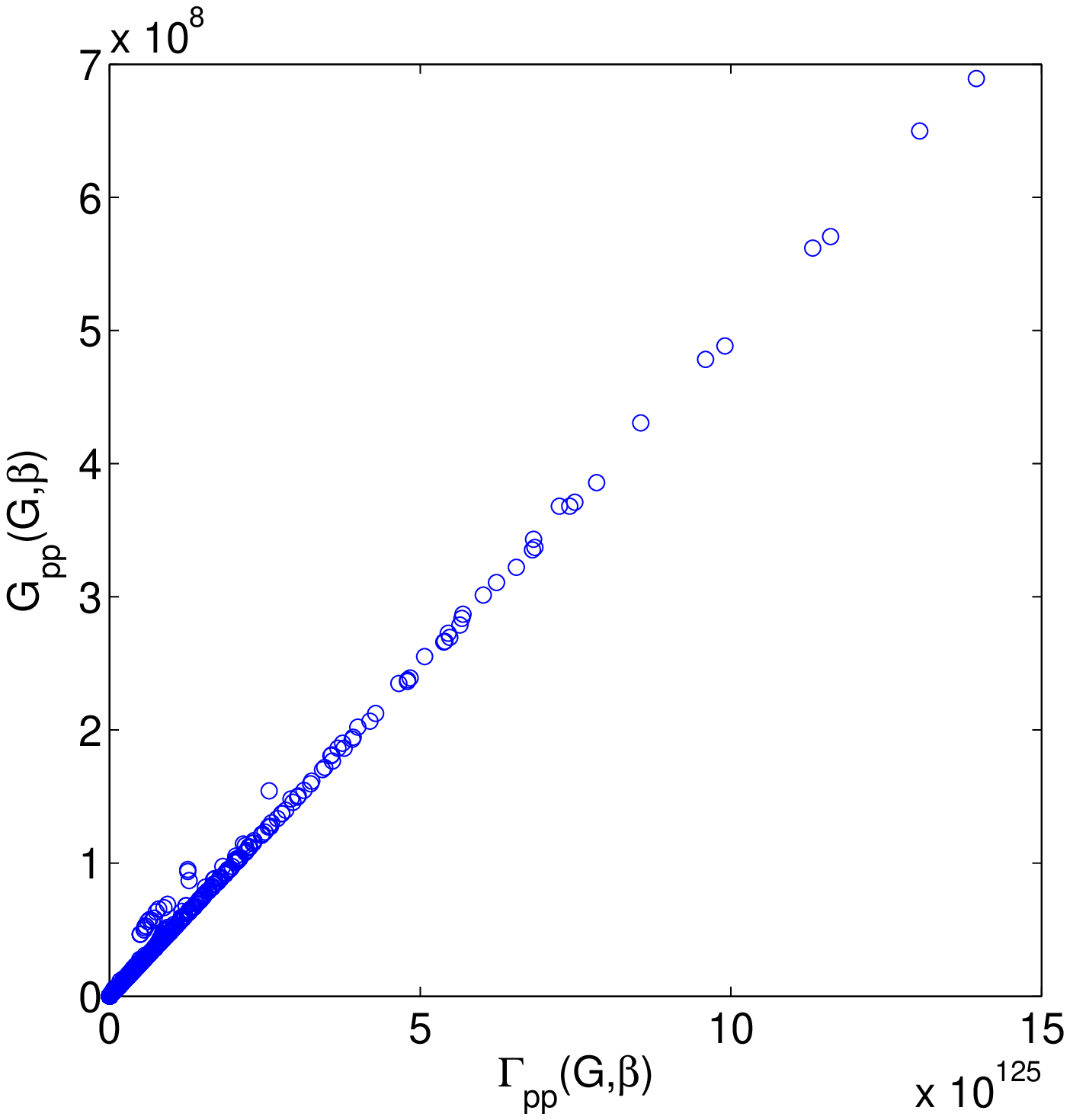}\includegraphics[width=0.5\textwidth]{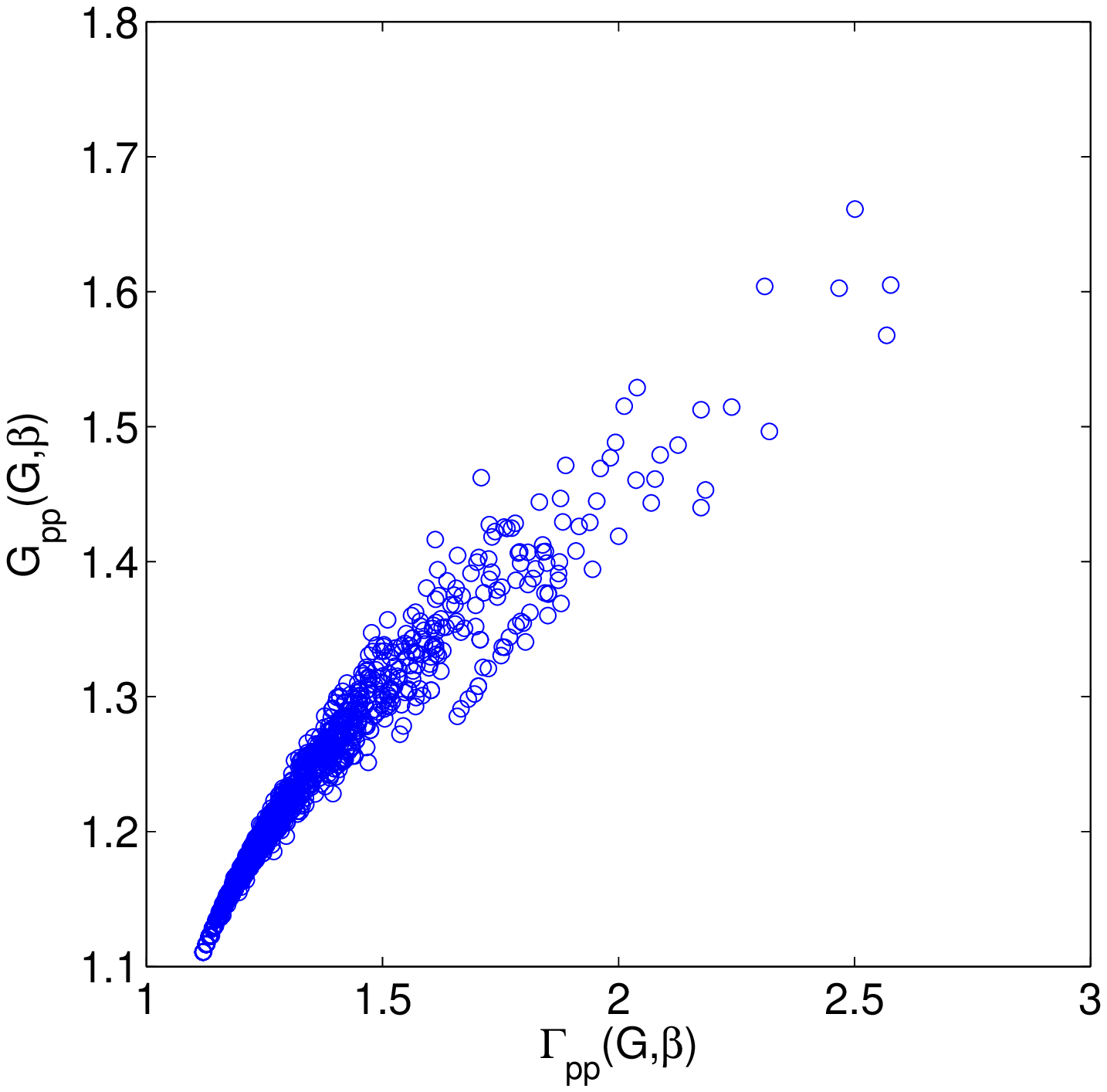}}
\par\end{centering}
\textcolor{black}{\caption{Correlation between the subgraph centrality based on the exponential
matrix function $G_{pp}\left(G,\beta\right)$ and on the matrix function
$D\left(A\right)$, $\Gamma_{pp}\left(G,\beta\right)$ for the protein-protein
interaction network of yeast (top plots) and the network of directors
in the corporate elite in US (bottom plots). The plots on the left
are for $\beta=1$ and on the right for $\beta=0.1$.}
}

\textcolor{black}{\label{subgraph centralities}}
\end{figure}

\textcolor{black}{In order to study how significant the differences
between $G_{pp}\left(G,\beta\right)$ and $\Gamma_{pp}\left(G,\beta\right)$
are for relevant network properties we consider the following problem.
We consider the identification of essential proteins in the protein-protein
interaction network of yeast \cite{PIN_1}. Essential proteins are
those for which if the corresponding gene is knocked out the entire
cell dies. Thus, they are considered to be essential for the survival
of the corresponding organisms. In this case we study how many essential
proteins exists in the top 10\% of proteins ranked according to a
given centrality index. The hypothesis behind this experiment is that
the most central proteins have higher probability of being essential.
Consequently, we rank all the proteins in the yeast PIN according
to $G_{pp}\left(G,\beta\right)$ and $\Gamma_{pp}\left(G,\beta\right)$
for $0\leq\beta\leq1$ with step $0.01$. We then select the top 10\%
of these proteins and count how many of them are essential. The results
for the two centrality indices considered here are illustrated in
Figure \ref{PIN yeast}(a) where it can be seen that both indices
reach the same maximum number of 115 essential proteins identified.
However, while $G_{pp}\left(G,\beta\right)$ reaches this maximum
for $0.47\leq\beta\leq0.57$, the maximum is reached by $\Gamma_{pp}\left(G,\beta\right)$
for $\beta=0.18$. In the Figure \ref{PIN yeast}(b) we illustrate
the receiving operating characteristic (ROC) for the classification
of the essential proteins using both indices. }

\textcolor{black}{}
\begin{figure}
\begin{centering}
\textcolor{black}{\includegraphics[width=0.5\textwidth]{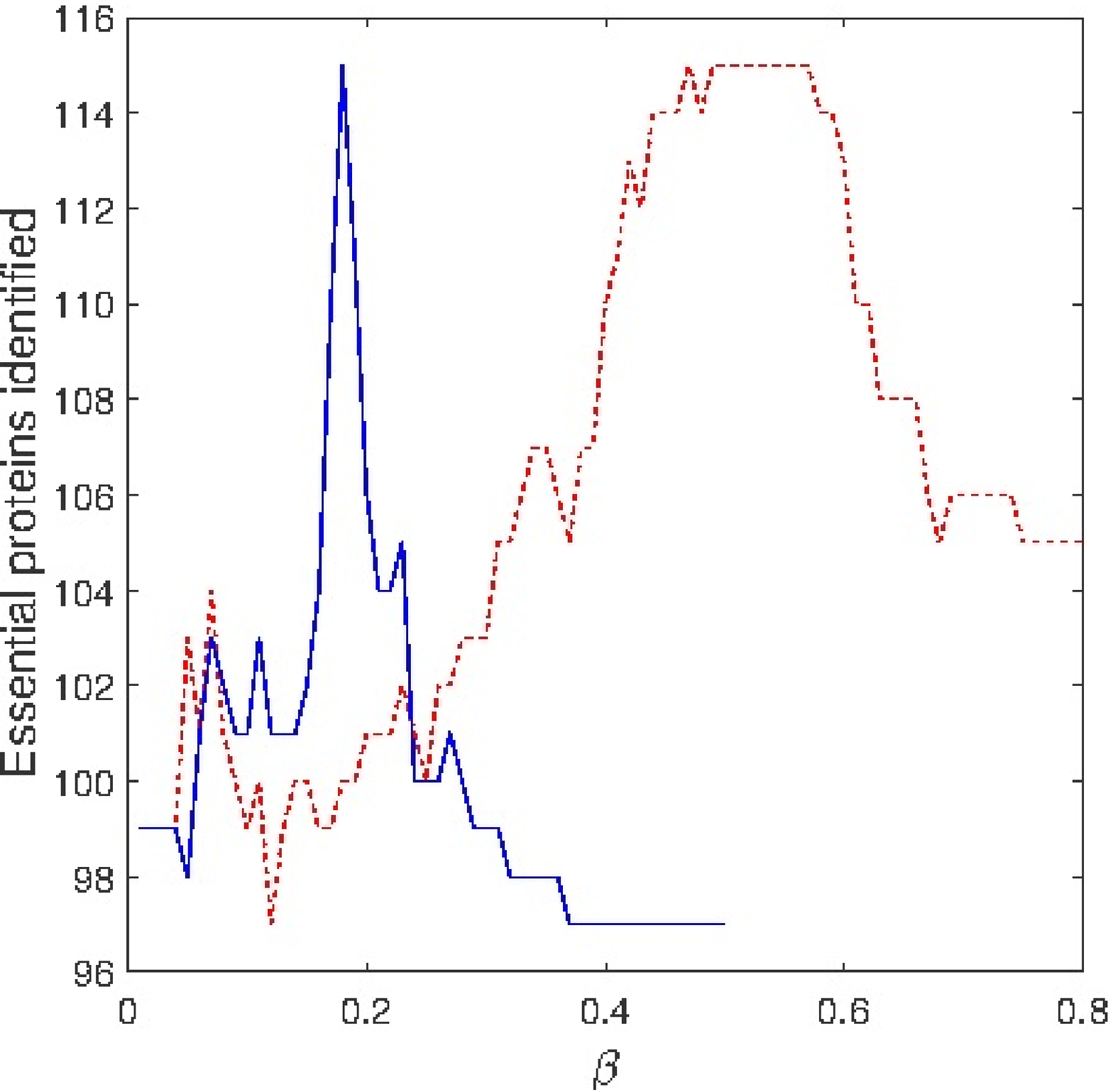}\includegraphics[width=0.5\textwidth]{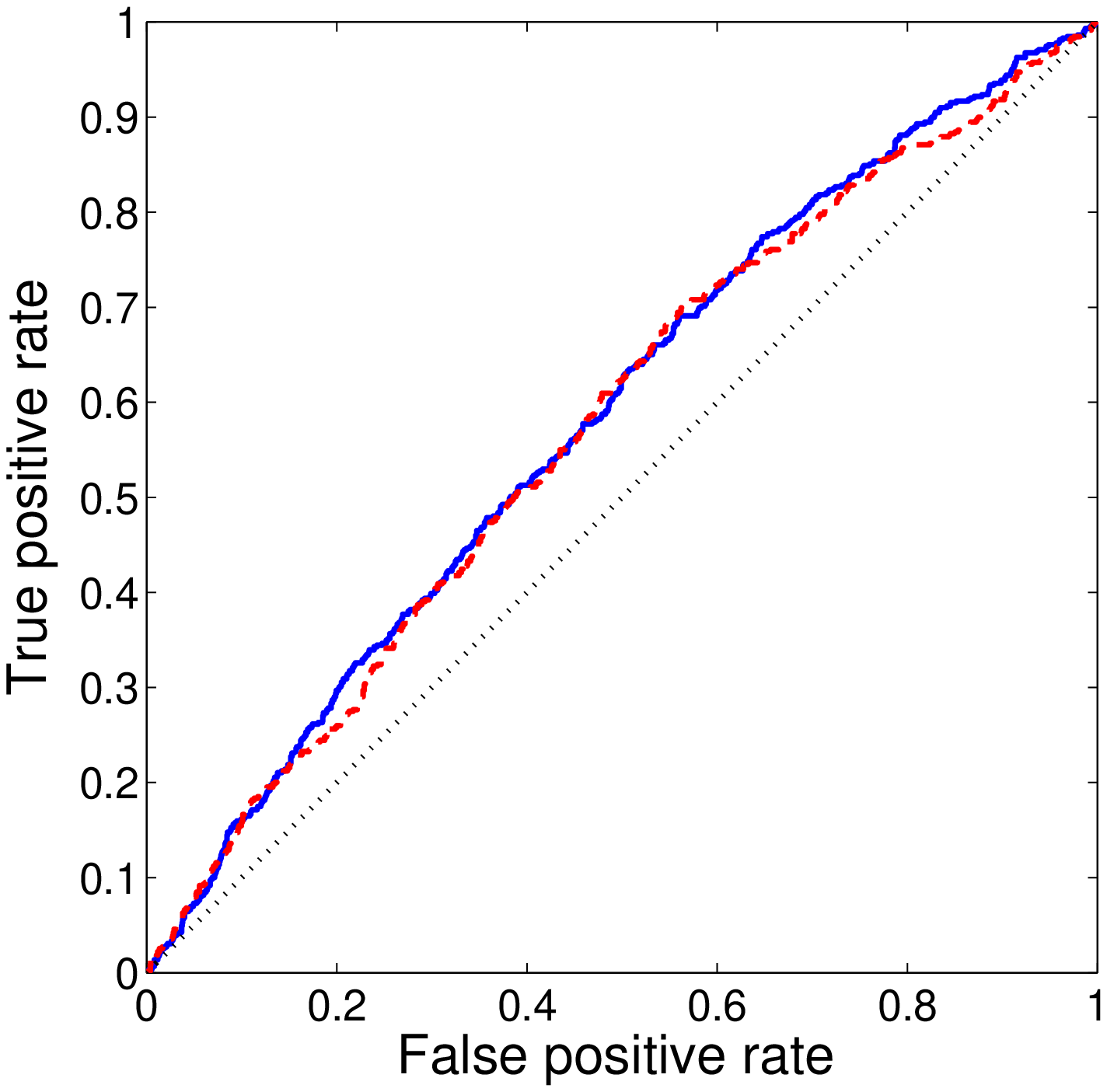}}
\par\end{centering}
\begin{centering}
\textcolor{black}{\caption{(a) Number of essential proteins identified by $G_{pp}\left(G,\beta=0.5\right)$
(red broken curve) and by $\Gamma_{pp}\left(G,\beta=0.18\right)$
(blue continuous line). (b) Illustration of the ROC curves for the
classification of essential proteins in yeast protein interaction
network (PPI) using $G_{pp}\left(G,\beta=0.5\right)$ (red broken
line) and $\Gamma_{pp}\left(G,\beta=0.18\right)$ (blue continuous
line).}
}
\par\end{centering}
\textcolor{black}{\label{PIN yeast}}
\end{figure}

\textcolor{black}{Apart from the visual similarities which are evident
from a simple inspection of the curves, the quantitative analysis
also indicates that there are no significant differences in the quality
of the classification using these indices. For instance, the area
below the curves for the classification of essential proteins in yeast
protein interaction network (PPI) using $G_{pp}\left(G,\beta=0.5\right)$
and $\Gamma_{pp}\left(G,\beta=0.18\right)$ are both 0.69. We can
see in (\ref{Correlindice}), the indices highly correlate for higher
values of $\beta_{1},\beta_{2}$ and also for small $\beta_{1},\beta_{2}$,
provided they are close together. In closing, there are no significant
differences in the quality of the classification models using $G_{pp}\left(G,\beta\right)$
and $\Gamma_{pp}\left(G,\beta\right)$ when the appropriate values
of $\beta$ are considered. For the sake of comparison we give the
values of essential proteins identified by other centrality indices:
eigenvector (97); degree (86); closeness (77); betweenness (71) and
a random ranking of the proteins identifies 52 essential proteins. }

\textcolor{black}{}
\begin{figure}
\begin{centering}
\textcolor{black}{\includegraphics[width=1\textwidth]{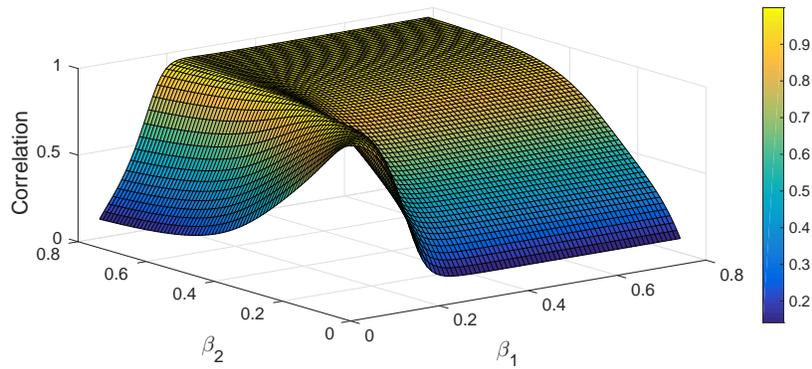}}
\par\end{centering}
\textcolor{black}{\caption{Correlation between the diagonal entries of the protein-protein interaction
network of yeast based on the exponential matrix function $G_{pp}\left(G,\beta_{1}\right)$
and on the matrix function $D\left(A\right)$, $\Gamma_{pp}\left(G,\beta_{2}\right)$
for $\beta_{1},\beta_{2}$ between $0$ and $0.75$ (step size of
$0.01$). }
}

\textcolor{black}{\label{Correlindice}}
\end{figure}

The main conclusion of this subsection is that in the case of networks
that do not contain significantly large chordless cycles, the double
factorial penalization of walks can produce similar results as the
indices using single-factorial one when the change of the parameter
$\beta$ is allowed. In the next section we will explore the differences
observed between these two schemes for networks containing significantly
large holes in their structures.

\subsection{\textcolor{black}{Networks with holes}}

\textcolor{black}{As we have analyzed in Section 5 the graphs containing
holes, i.e., chordless cycles, in their structures display a different
ranking of the nodes according to the indices developed from the single-
and the double-factorial penalization of walks. This situation is
frequently observed in real-world networks. A couple of very typical
examples are the}\textit{\textcolor{black}{{} urban street networks}}\textcolor{black}{{}
\cite{spatial networks} and the }\textit{\textcolor{black}{protein
residue networks}}\textcolor{black}{{} (see for instance \cite{EstradaBook}).
In the first case, the streets of a city are represented as the edges
of the network and their intersections are represented by the nodes.
In the second case, the nodes represent the $\alpha$-carbons of the
amino acids in the protein and two nodes are connected if the corresponding
amino acids are at a distance that allows their physical interaction.
In urban street networks\textemdash see Figure \ref{city streets}\textemdash the
holes are regions without streets, such as parks, big stores or natural
environments like ponds. In the protein residue networks the holes
are binding sites\textemdash regions in which amino acids are spatially
separated to allocate other molecules\textemdash for small organic
molecules or other proteins. A notable difference between the two
systems is that while the first are represented mainly by planar graphs,
the second is represented mainly by nonplanar ones. The determination
of holes in networks is not a trivial problem and many efforts are
directed to this goal due to the importance of these topological features
in real-world systems \cite{holes in graphs,chordless,homotopy}.
Information about whether the network contains holes or not can be
obtained by means of the so-called ``spectral scaling method'' \cite{spectral scaling,structural classes}. }

\textcolor{black}{}
\begin{figure}
\begin{centering}
\textcolor{black}{\includegraphics[width=0.45\textwidth]{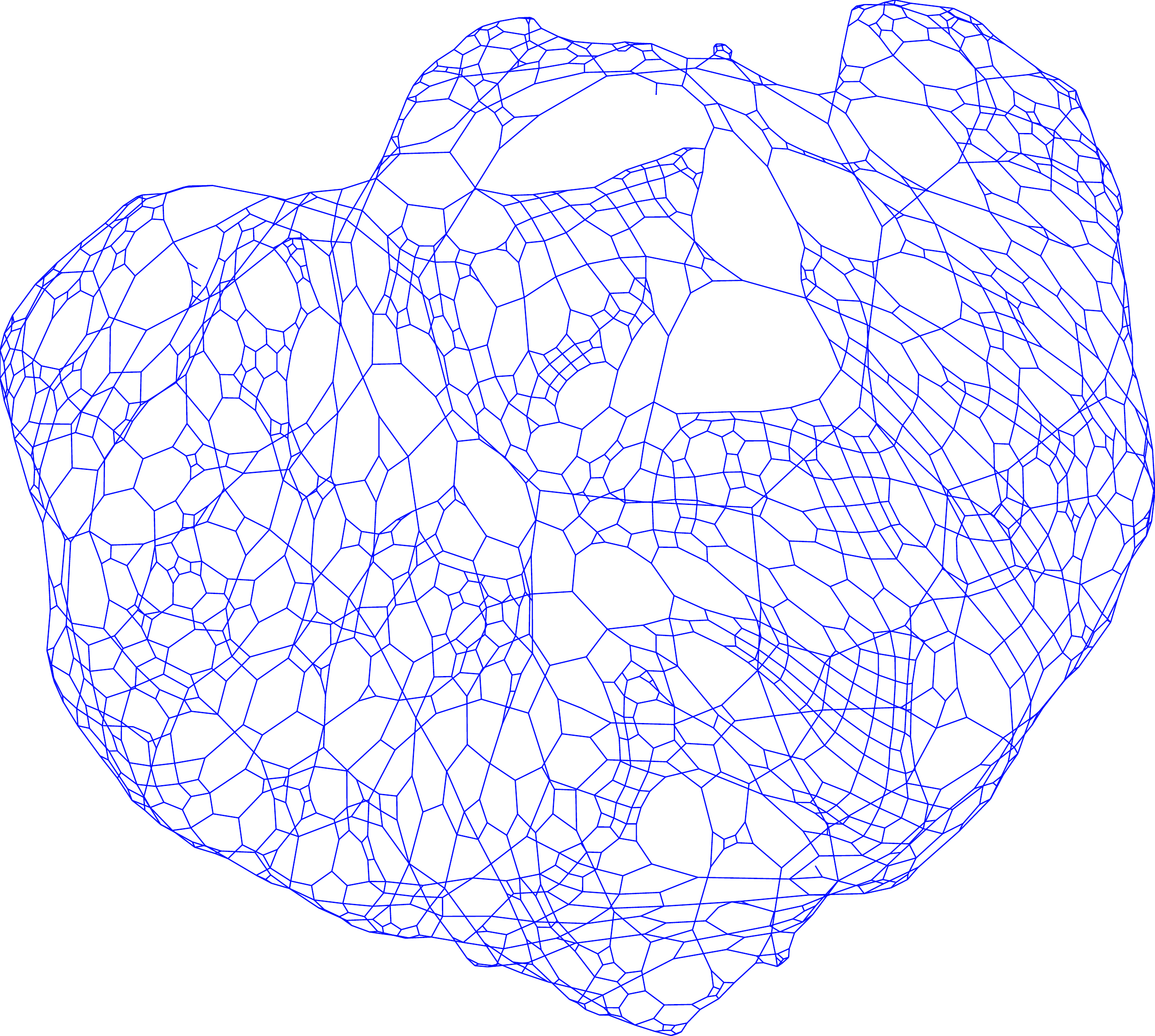}\includegraphics[width=0.55\textwidth]{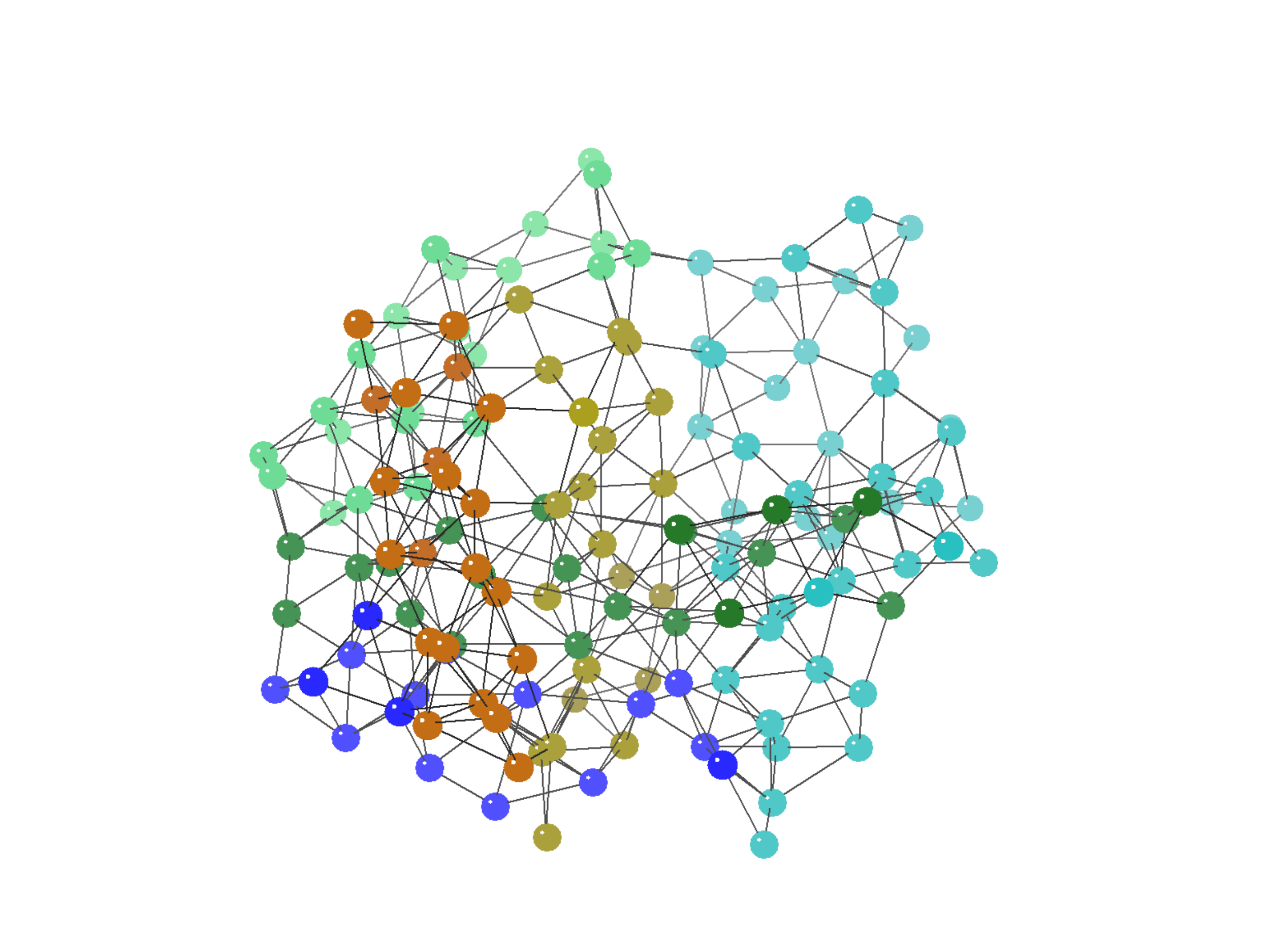}}
\par\end{centering}
\textcolor{black}{\caption{Illustration of the urban street network of the city centres of Bologna
(left) and a protein residue network (right). Both networks contain
chordless cycles (holes) although the first is planar and the second
is nonplanar.}
}

\textcolor{black}{\label{city streets}}
\end{figure}

\textcolor{black}{In this Section we compare the ranking of nodes
using the subgraph centrality based on single and double-factorial
penalization of the walks for a series of urban street networks as
well as protein residue networks. For the urban street networks we
selected 14 cities from around the World as a representative set.
For the protein residue network we selected 14 proteins whose structures
have been obtained from x-ray crystallography and deposited in the
}\textit{\textcolor{black}{Protein Data Bank}}\textcolor{black}{{} (PDB)
\cite{PDB}. Each protein is identified with a unique code of one
number and three lowercase letters. A fourth letter sometimes appears
to designate the name of the chain to which the protein belongs to.
We selected proteins of different domains and sizes ranging from 100
to 1,000 amino acids. In order to compare the ranking based on both
approaches, i.e., single- and double-factorial subgraph centralities,
we use the }\textit{\textcolor{black}{Spearman rank correlation coefficient}}\textcolor{black}{.
This index indicates how correlated are the ranking of the nodes are
according to both indices.}

\textcolor{black}{In Table \ref{real networks} we give the values
of the Spearman rank correlation coefficient for all the networks
analyzed in this work. For the urban street networks the average rank
correlation coefficient is $0.727$ and for the protein residue networks
it is $0.747$. In both cases the rank correlation indicates that
the two indices rank the nodes in significantly different ways. For
the sake of comparison the Spearman correlation coefficient between
the two indices for the network of Internet as an Autonomous System
is $0.9999$. This finding\textemdash that both indices are not highly
rank-correlated\textemdash is very important because it indirectly
indicates that the double-factorial penalization of walks adds some
structural characteristics not described by the single-factorial indices.
It is important to remark that in the urban street networks there
are rank correlation coefficients as low as $0.56$ and as high as
$0.92$. However, in the protein residue networks the rank correlations
are less deviated from the mean. These differences reflect the important
fact that cities are in general more heterogeneous than proteins.
That is, there are cities with clearly defined holes in their structures
while others do not necessarily display such topological features.
However, we have previously shown that 95\% of representative proteins
contain holes indicating that the presence of chordless cycles is
a universal property in these systems \cite{universality}.}

\textcolor{black}{}
\begin{table}
\begin{centering}
\textcolor{black}{}%
\begin{tabular}{|c|>{\centering}p{1cm}|>{\centering}p{2cm}|c|c|c|>{\centering}p{2cm}|}
\hline 
\textcolor{black}{City network} & \textcolor{black}{$n$} & \textcolor{black}{Rank Correlation} &  & \textcolor{black}{PDB} & \textcolor{black}{$n$} & \textcolor{black}{Rank Correlation}\tabularnewline
\hline 
\hline 
\textcolor{black}{Ahmedabad} & \textcolor{black}{4874} & \textcolor{black}{0.6002} &  & \textcolor{black}{1ccr} & \textcolor{black}{111} & \textcolor{black}{0.7296}\tabularnewline
\hline 
\textcolor{black}{Atlanta} & \textcolor{black}{3234} & \textcolor{black}{0.8562} &  & \textcolor{black}{1cpq} & \textcolor{black}{129} & \textcolor{black}{0.7846}\tabularnewline
\hline 
\textcolor{black}{Barcelona} & \textcolor{black}{5575} & \textcolor{black}{0.9151} &  & \textcolor{black}{1berA} & \textcolor{black}{199} & \textcolor{black}{0.8215}\tabularnewline
\hline 
\textcolor{black}{Berlin} & \textcolor{black}{4495} & \textcolor{black}{0.6490} &  & \textcolor{black}{1bpyA} & \textcolor{black}{326} & \textcolor{black}{0.7630}\tabularnewline
\hline 
\textcolor{black}{Bologna} & \textcolor{black}{825} & \textcolor{black}{0.9144} &  & \textcolor{black}{1cem} & \textcolor{black}{363} & \textcolor{black}{0.8387}\tabularnewline
\hline 
\textcolor{black}{Cambridge} & \textcolor{black}{1509} & \textcolor{black}{0.6513} &  & \textcolor{black}{1chm} & \textcolor{black}{401} & \textcolor{black}{0.7196}\tabularnewline
\hline 
\textcolor{black}{Chengkan} & \textcolor{black}{414} & \textcolor{black}{0.7127} &  & \textcolor{black}{1bmfA} & \textcolor{black}{487} & \textcolor{black}{0.7357}\tabularnewline
\hline 
\textcolor{black}{Hong Kong} & \textcolor{black}{916} & \textcolor{black}{0.7417} &  & \textcolor{black}{1ctn} & \textcolor{black}{538} & \textcolor{black}{0.7272}\tabularnewline
\hline 
\textcolor{black}{Mecca} & \textcolor{black}{1464} & \textcolor{black}{0.6943} &  & \textcolor{black}{1aorA} & \textcolor{black}{605} & \textcolor{black}{0.7926}\tabularnewline
\hline 
\textcolor{black}{Milton Keynes} & \textcolor{black}{5581} & \textcolor{black}{0.6463} &  & \textcolor{black}{1cyg} & \textcolor{black}{680} & \textcolor{black}{0.7569}\tabularnewline
\hline 
\textcolor{black}{Oxford} & \textcolor{black}{1622} & \textcolor{black}{0.6951} &  & \textcolor{black}{8acn} & \textcolor{black}{753} & \textcolor{black}{0.7754}\tabularnewline
\hline 
\textcolor{black}{Penang} & \textcolor{black}{7055} & \textcolor{black}{0.5593} &  & \textcolor{black}{1qba} & \textcolor{black}{863} & \textcolor{black}{0.6874}\tabularnewline
\hline 
\textcolor{black}{Rotterdam} & \textcolor{black}{1300} & \textcolor{black}{0.6472} &  & \textcolor{black}{1alo} & \textcolor{black}{908} & \textcolor{black}{0.6494}\tabularnewline
\hline 
\textcolor{black}{Yuliang} & \textcolor{black}{88} & \textcolor{black}{0.8990} &  & \textcolor{black}{1bglA} & \textcolor{black}{1021} & \textcolor{black}{0.6786}\tabularnewline
\hline 
\end{tabular}
\par\end{centering}
\textcolor{black}{\caption{Spearman rank correlation coefficients for the subgraph centrality
obtained with the single and double-factorial penalization of walks.
The results are for the urban street networks (left) and protein residue
networks (right) studies here.}
}

\textcolor{black}{\label{real networks}}
\end{table}

\textcolor{black}{In order to illustrate the differences in the ranking
of nodes in a network using both types of centrality indices we consider
here the urban street network of Cambridge in the UK and the protein
residue network of the protein 1qba. In Figure \ref{Cambridge} we
plot the subgraph centralities of each node in the Cambridge urban
street network (top images) and of the protein residue network (bottom
images) with radius proportional to the logarithm of the subgraph
centrality based on single- (left image) and double-factorial (right
image). The logarithm is used here to avoid be fooled by the very
large numbers of the subgraph centrality in these networks. As can
be seen there are significant differences in the centrality of the
nodes based on both approaches. The main difference is that the centrality
based on the single-factorial identifies fewer hubs than the double-factorial,
which is able to delineate complete regions in the networks. }

\textcolor{black}{}
\begin{figure}
\begin{centering}
\textcolor{black}{\includegraphics[width=0.5\textwidth]{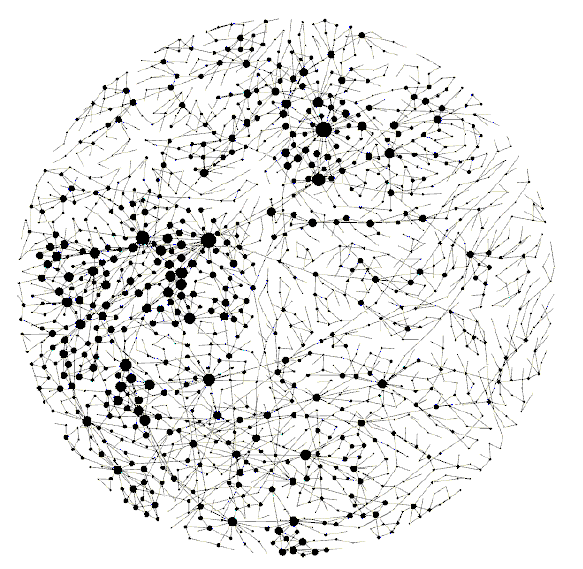}\includegraphics[width=0.5\textwidth]{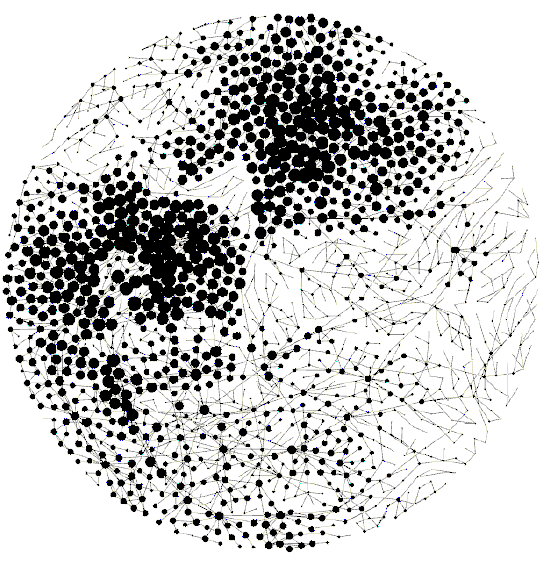}}
\par\end{centering}
\begin{centering}
\textcolor{black}{\includegraphics[width=0.5\textwidth]{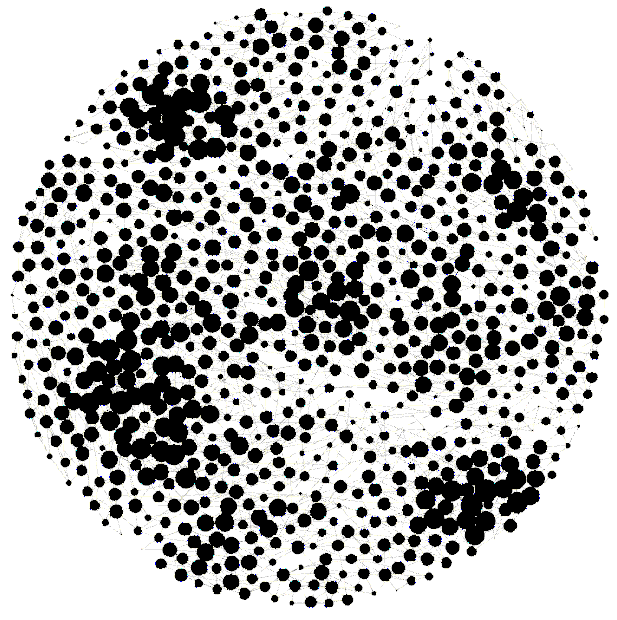}\includegraphics[width=0.5\textwidth]{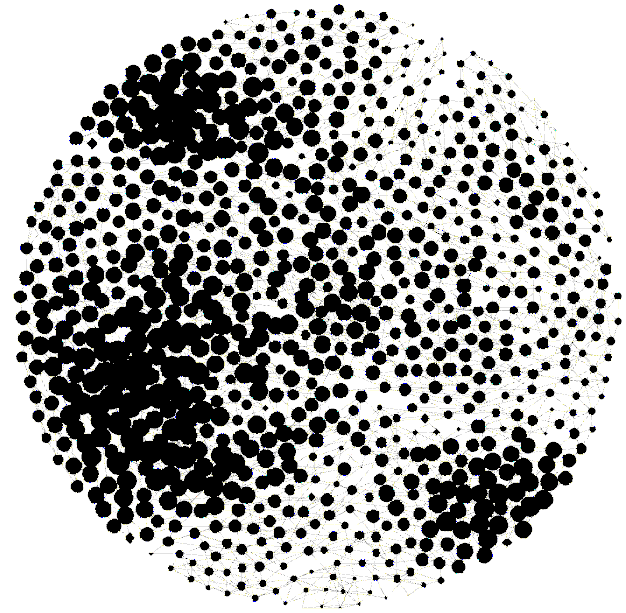}}
\par\end{centering}
\textcolor{black}{\caption{Illustration of the subgraph centrality of nodes in the urban street
network of Cambridge, UK. The radius of the nodes is proportional
to the logarithm of the subgraph centrality based on single- (left
image) and double-factorial (right image) penalization of the walks.}
}

\textcolor{black}{\label{Cambridge}}
\end{figure}

\textcolor{black}{Here we have confirmed what we have first analyzed
in Section 5, that in networks containing holes, the subgraph centrality
and other indices based on walks are significantly different when
the single- or double-factorial penalization are used. In general,
we observe significantly different ranking of the nodes, with the
subgraph centrality based on double-factorial identifying a larger
number of central nodes than the one based on single-factorial penalization.
These findings are important for the analysis of specific network
problems in particular areas of research, as holes mean different
things in different contexts. }

\section{\textcolor{black}{Conclusion}}

\textcolor{black}{We have introduced here a new matrix function for
studying graphs and real-world networks. This new matrix function
of the adjacency matrix of a graph is based on the double-factorial
penalization of walks between nodes in a graph. We have observed here
that there are two groups of networks for which the behavior of the
indices based on the double-factorial penalization changes with respect
to that of the single-factorial one. In the first case we have considered
networks where there are no structural holes, such as the case of
the protein-protein interaction networks of yeast. In this situation
we have observed that by introducing a weighting scheme of the form
$c_{k}=\frac{1}{k!}$ or $c_{k}=\frac{1}{k!!}$ the double factorial
indices produce similar results as the single-factorial one for the
identification of essential proteins in the yeast PIN. }

\textcolor{black}{The second group of networks is formed by those
containing significant chordless cycles or holes in their structures,
such as urban street networks and protein residue networks. In those
cases the contribution of long walks is very important, in particular
for navigating around such long holes in the network. In those cases
we have shown how a centrality index based on single- as well as on
the double-factorial penalization of the walks produce significant
differences in the ranking of the nodes. We should stress that significantly
large chordless cycles are present in a variety of networks and that
their study is of major importance in communication systems, where
they should be avoided as regions of zero-coverage of the communication
signals. Consequently, the new scheme of penalizing walks by using
the double-factorial opens new possibilities for the study of many
problems in real-world networks.}

\section{\textcolor{black}{Acknowledgment}}

\textcolor{black}{EE thanks the Royal Society of London for a Wolfson
Research Merit Award. GS thanks EPRSC for a PhD scholarship. The authors
thank both referees for excellent revision and constructive comments
on the manuscript.}


\begin{thebibliography}{10}
\bibitem{Bessel}M. Abramovich, I. Stegun, Handbook of Mathematical
Functions with Formulas, graphs and mathematical Tables, National
Bureau of Standarts, Appl. Math. series 55 (1964).

\textcolor{black}{\bibitem{Benzi 3}F. Arrigo, M. Benzi, Updating
and downdating techniques for optimizing network communicability,
SIAM J. Sci. Comp. 38 (2016) B25\textendash B49.}

\textcolor{black}{\bibitem{spatial networks}M. Barthélemy, Spatial
networks, Phys. Rep. 499 (2011) 1\textendash 101.}

\textcolor{black}{\bibitem{Benzi 2}M. Benzi, C. Klymko, On the limiting
behavior of parameter-dependent network centrality measures, SIAM
J. Matrix Anal. Appl. 36 (2015) 686\textendash 706.}

\textcolor{black}{\bibitem{Benzi 1}M. Benzi, E. Estrada, C. Klymko,
Ranking hubs and authorities using matrix functions, Linear Algebra
Appl. 438 (2013) 2447\textendash 2474.}

\textcolor{black}{\bibitem{PDB}H. M. Berman, J. Westbrook, Z. Feng,
G. Gilliland, T. N. Bhat, H. Weissig, I. N. Shindyalov, P. E. Bourne,
The protein data bank, Nucleic Acids Res. 28 (2000) 235\textendash 242.}

\textcolor{black}{\bibitem{holes in graphs}N. Chandrasekharan, V.
S. Lakshmanan, M. Medidi, Efficient parallel algorithms for finding
chordless cycles in graphs, Parallel Proc. Lett. 3 (1993) 165\textendash 170.}

\textcolor{black}{\bibitem{LucianoReview}L. F. Costa, O. N. Oliveira
Jr, G. Travieso, F. A. Rodrigues, and P. R. Villas Boas, L. Antiqueira,
M. P. Viana, and L. E. Correa Rocha. Analyzing and modeling real-world
phenomena with complex networks: a survey of applications, Adv. Phys.
60 (2011) 329\textendash 412.}

\textcolor{black}{\bibitem{Brain networks}J.J. Crofts, D.J. Higham,
A weighted communicability measure applied to complex brain networks,
J. Roy. Soc. Interface 6 (2009) 411\textendash 414.}

\textcolor{black}{\bibitem{brain networks_2}J.J. Crofts, D.J. Higham,
R. Bosnell, S. Jbabdi, P.M. Matthews, T.E.J. Behrens, H. Johansen-Berg,
Network analysis detects changes in the contralesional hemisphere
following stroke, Neuroimage 54 (2011) 161\textendash 169.}

\textcolor{black}{\bibitem{Estrada index 1}H. Deng, S. Radenkovi\'{c},
I. Gutman, The Estrada index, in: D. Cvetkovi\'{c}, I. Gutman (Eds.),
Applications of Graph Spectra, Math. Inst, Belgrade, 2009, pp. 123\textendash 140. }

\textcolor{black}{\bibitem{chordless}E. S. Dias, D. Castonguay, H.
Longo, W. A. R. Jradi, Efficient enumeration of all chordless cycles
in graphs, Comput. Res. Repos. abs/1309.1051 (2013).}

\textcolor{black}{\bibitem{fullerene}T. Došli\'{c}, Bipartivity of
fullerene graphs and fullerene stability, Chem. Phys. Lett. 412 (2008)
336\textendash 340.}

\textcolor{black}{\bibitem{Ejov_1}V. Ejov, J.A. Filar, S.K. Lucas,
P. Zograf, Clustering of spectra and fractals of regular graphs, J.
Math. Anal. Appl. 333 (2007) 236\textendash 246. }

\textcolor{black}{\bibitem{Ejov_2} V. Ejov, S. Friedlan, G.T. Nguyen,
A note on the graph\textquoteright s resolvent and the multifilar
structure, Linear Algebra Appl. 431 (2009) 1367\textendash 1379.}

\textcolor{black}{\bibitem{Estrada Higham}E. Estrada, D. J. Higham,
Network properties revealed through matrix functions, SIAM Rev.}\textit{\textcolor{black}{{}
}}\textcolor{black}{52 (2010) 96\textendash 714.}

\textcolor{black}{\bibitem{Subgraph Centrality}E. Estrada, J.A. Rodríguez-Velázquez,
Subgraph centrality in complex networks,}\textit{\textcolor{black}{{}
}}\textcolor{black}{Phys. Rev. E 71}\textbf{\textcolor{black}{{} }}\textcolor{black}{(2005)
056103.}

\textcolor{black}{\bibitem{Communicability}E. Estrada, N. Hatano,
Communicability in complex networks}\textit{\textcolor{black}{, }}\textcolor{black}{Phys.
Rev. E 77 (2008) 036111.}

\textcolor{black}{\bibitem{Temperature}E. Estrada, N. Hatano, Statistical-mechanical
approach to subgraph centrality in complex networks, Chem. Phys. Lett.
439 (2007) 247\textendash 251. }

\textcolor{black}{\bibitem{Protein folding}E. Estrada, Characterization
of the folding degree of proteins, Bioinformatics 18 (2002) 697\textendash 704. }

\textcolor{black}{\bibitem{Zooming in and out} E. Estrada, Generalized
walks-based centrality measures for complex biological networks, J.
Theor. Biol. 263 (2010) 556\textendash 565.}

\textcolor{black}{\bibitem{airlines}E. Estrada, J. Gómez-Gardeñes,
Network bipartivity and the transportation efficiency of European
passenger airlines, Physica D 323-324 (2016) 57\textendash 63.}

\textcolor{black}{\bibitem{Estrada Hatano Benzi}E. Estrada, N. Hatano,
M. Benzi, The physics of communicability in complex networks,}\textit{\textcolor{black}{{}
}}\textcolor{black}{Phys. Rep. 514 (2012) 89\textendash 119.}

\textcolor{black}{\bibitem{PIN_2} E. Estrada, Protein bipartivity
and essentiality in the yeast protein\textendash protein interaction
network, J. Proteome Res. 5 (2006) 2177\textendash 2184.}

\textcolor{black}{\bibitem{EstradaBook}E. Estrada, }\textit{\textcolor{black}{The
Structure of Complex Networks. Theory and Applications}}\textcolor{black}{,
Oxford University Press, 2011. }

\textcolor{black}{\bibitem{PIN_1}E. Estrada, Virtual identification
of essential proteins within the protein interaction network of yeast,
Proteomics 6 (2006) 35\textendash 40.}

\textcolor{black}{\bibitem{spectral scaling}E. Estrada, Spectral
scaling and good expansion properties in complex networks, Europhys.
Lett. 73 (2006) 649.}

\textcolor{black}{\bibitem{structural classes}E. Estrada, Topological
structural classes of complex networks, Phys. Rev. E 75 (2007) 016103.}

\textcolor{black}{\bibitem{universality}E. Estrada, Universality
in protein residue networks, Biophys. J. 98 (2010) 890\textendash 900.}

\textcolor{black}{\bibitem{homotopy}Q. Fang, J. Gao, L. J. Guibas,
Locating and bypassing holes in sensor networks, Mobile Net. Appl.
11 (2006) 187\textendash 200.}

\textcolor{black}{\bibitem{DF}H.W. Gould, J. Quaintance, Double fun
with double factorials, Mathematics Magazine 85 (2012) 177\textendash 192,
.}

\textcolor{black}{\bibitem{Estrada index 2}I. Gutman, H. Deng, and
S. Radenkovi\'{c}, The Estrada index: an updated survey, in: D. Cvetkovi\'{c},
I. Gutman (Eds.), Selected Topics on Applications of Graph Spectra,
Math. Inst., Beograd, 2011, pp. 155\textendash 174.}

\textcolor{black}{\bibitem{Function of matrices}N. J. Higham, }\textit{\textcolor{black}{Functions
of Matrices: Theory and Computation}}\textcolor{black}{, Society for
Industrial and Applied Mathematics, Philadelphia, PA, 2008.}

\textcolor{black}{\bibitem{Katz centrality}L. Katz, A new index derived
from sociometric data analysis, Psychometrika 18 (1953) 39\textendash 43.}

\textcolor{black}{\bibitem{Newman review}M. E. J. Newman, The structure
and function of complex networks, SIAM Rev. 45 (2003) 167\textendash 256.}

\textcolor{black}{\bibitem{Estrada index}J.A. de la Peña, I. Gutman,
J. Rada, Estimating the Estrada index, Linear Algebra Appl. 427 (2007)
70\textendash 76.}

\textcolor{black}{\bibitem{erf}R. Taylor, Solution of the linearized
equations of multicomponent mass transfer, Ind. Eng. Chem. Fundam.
21 (1982) 407\textendash 413.}

\textcolor{black}{\bibitem{Tordesillas}D.M. Walker, A. Tordesillas,
Topological evolution in dense granular materials: a complex networks
perspective, Int. J. Solids Struct. 47 (2010) 624\textendash 639.}

\textcolor{black}{\bibitem{Natural connectivity 1}J. Wu, M. Barahona,
Y-J. Tan, H-Z. Deng, Robustness of regular ring lattices based on
natural connectivity, Int. J. Syst. Sci. 42 (2011) 1085\textendash 1092.}

\textcolor{black}{\bibitem{Natural connectivity 2}J. Wu, H.-Z. Deng,
Y.-J. Tan, D.-Z. Zhu, Vulnerability of complex networks under intentional
attack with incomplete information, J. Phys. A: Math. Theor. 40 (2007)
2665.}
\end{thebibliography}
\end{document}